\documentclass[11pt]{article}
\usepackage{natbib}
\usepackage{amsmath}
\usepackage{amsfonts,amssymb,amsthm}
\usepackage{setspace}
\usepackage[toc,page]{appendix}
\usepackage{upgreek}
\usepackage{float}
\restylefloat{table}
\usepackage{multirow}
\usepackage{graphicx}

\theoremstyle{definition}
\newtheorem{defn}{Definition}[section]

\theoremstyle{plain}

\newtheorem{thm}[defn]{Theorem}
\newtheorem{cor}[defn]{Corollary}
\newtheorem{prop}[defn]{Proposition}
\newtheorem{lem}[defn]{Lemma}
\newtheorem{ex}[defn]{Example}

\theoremstyle{remark}

\newtheorem{rem}{Remark}

\numberwithin{equation}{section}
\usepackage[labelfont=bf]{caption}
\usepackage[colorinlistoftodos]{todonotes}
\usepackage[colorlinks=true, allcolors=blue]{hyperref}

\newcommand{\tr}[0]{\text{tr}}
\newcommand\me{\mathrm{e}}
\newcommand\mb{\mathbb}
\newcommand\mE{\mathbb{E}}
\newcommand\mP{\mathbb{P}}
\newcommand\mF{\mathcal{F}}

\usepackage[margin=1in]{geometry}
\begin{document}



\title{A Bernstein-type Inequality for High Dimensional Linear Processes with Applications to Robust Estimation of Time Series Regressions}
\author{Linbo Liu$^{1,2}$\thanks{Work done prior to joining AWS AI Labs. Email: linbol@amazon.com}~ and Danna Zhang$^1$\thanks{Email: daz076@ucsd.edu}\\ $^1$Department of Mathematics, UC San Diego,\quad $^2$AWS AI Labs}
\date{}
\maketitle

\begin{abstract}
Time series regression models are commonly used in time series analysis. However, in modern real-world applications, serially correlated data with an ultra-high dimension and fat tails are prevalent. This presents a challenge in developing new statistical tools for time series analysis. In this paper, we propose a novel Bernstein-type inequality for high-dimensional linear processes and apply it to investigate two high-dimensional robust estimation problems: (1) time series regression with fat-tailed and correlated covariates and errors, and (2) fat-tailed vector autoregression. Our proposed approach allows for exponential increases in dimension with sample size under mild moment and dependence conditions, while ensuring consistency in the estimation process.
\vspace{9pt}

\end{abstract}
\section{Introduction}
High dimensional data analysis has become increasingly important in the era of big data, due to the explosion of massive datasets. High-dimensional linear regression has acquired significant relevance and attention. Specifically, consider the linear regression models
\begin{equation*}
Y_i = X_i^\top \beta + \xi_i, \quad i=1,\ldots, n
\end{equation*}
where $Y_i$, $X_i$ and $\xi_i$ are the response, covariate and error variables, respectively. Various regularization methods have been widely used for estimating the $p$-dimensional regression parameter vector, including \cite{tibshirani1996regression}, \cite{zou2005regularization}, \cite{fan2001variable},  \cite{bickel2009simultaneous}, \cite{meinshausen2009lasso} 
and many others; see \cite{buhlmann2011statistics} for 
a comprehensive overview. Most investigations assume that the covariates $X_i$ (if it is a random design) and errors $\xi_i$ are i.i.d.~Gaussian or sub-Gaussian random variables, which can be too restrictive in many real-world applications.

On one hand, serial correlation might occur when the data are collected over time. Linear regression with time series regressors and autoregressive errors is often considered (\cite{harvey1990econometric}, \cite{tsay1984regression}, \cite{shumway2000time}). On the other hand, many applications involving time series data are concerned with high dimensional objects and fat-tailed distributions, ranging from quantitative finance (\cite{cont2001empirical}) and portfolio allocation (\cite{kim2012measuring}) to risk management(\cite{koopman2008non}), brain network
(\cite{friston2011functional}) and geophysical dynamic studies (\citet{kondrashov2005hierarchy}).

Previous literature has made progress in linear regression with correlated errors. Specifically, the Lasso estimator was studied for linear regression with autoregressive errors by \cite{wang2007regression} and \cite{yoon2013penalized}, weakly dependent errors by \cite{gupta2012note} and long memory errors by \cite{kaul2014lasso}. However, these investigations mainly focus on cases where the dimension $p$ is smaller than the sample size $n$ or where the Gaussian assumption is imposed on the error process. More recently, \cite{wu2016performance} and \cite{chernozhukov2021lasso} used the framework of functional dependence measures to account for both dependent covariates and errors in linear regression, allowing $p$ to increase with $n$ at a polynomial rate while maintaining consistency. However, a narrow range is still restricted for the dimension in the presence of non-Gaussian and dependent errors. To address the ultra-high dimensional cases where $p$ can grow exponentially with $n$, various robust methods have been investigated for linear regression with i.i.d.~fat-tailed errors, including penalized Huber $M$-estimation (\cite{fan2017estimation}, \cite{loh2017statistical,loh2021scale}), sparse least trimmed squares (\cite{alfons2013sparse}), 
ESL-Lasso (\cite{wang2013robust}), among others. In this paper, we aim to consider robust estimation of time series regression allowing ultra-high dimensions and fat-tailed and correlated errors.

Vector autoregression (VAR) is another widely used linear model to describe the evolution of a set of variables over time. In recent years, there has been significant progress in estimating high-dimensional VAR models. Inspired from the development in high-dimensional linear regression, \cite{hsu2008subset}, \cite{nardi2011autoregressive} and \cite{basu2015regularized} considered the Lasso estimator using $\ell_1$ penalty. \cite{kock2015oracle} established oracle inequalities for high-dimensional vector autoregressive models. \cite{han2015direct} adopted a Dantzig-type penalization. \cite{guo2016high} proposed a Bayesian information criterion based on residual sums of least squares estimator to estimate high dimensional banded autoregression. However, most of these studies require the Gaussian assumption or the existence of finite exponential moment. In econometric analysis, \cite{sims1980macroeconomics} raised the concern that fat tails in VAR models can affect the validity of statistical inference and may lead to low degrees of freedom due to the estimation of possibly a large number of parameters. Therefore, there is a need to investigate robust estimation methods for high-dimensional fat-tailed VAR models.

In summary, our work will focus on tackling the challenges posed by high dimensional time series analysis with time series covariates, possibly correlated errors, fat tails, and ultra high dimension. This requires the development of new statistical tools that are tailored to the specific characteristics of these datasets. One of our key contributions is a novel Bernstein-type inequality for the sum of a bounded transformation of high dimensional linear processes. This inequality will be instrumental in obtaining consistent estimators under mild conditions, such as $\log p = o(n^c)$ for some $c>0$.


The paper is organized as follows. In Section 2, we introduce the framework of high dimensional linear processes and the important quantities that can characterize temporal and cross-sectional dependence. We then present a new Bernstein type inequality for high dimensional linear processes. In Section 3, we investigate two robust estimation problems: (1) time series linear regression with correlated and fat-tailed covariates and errors, and (2) autoregressive models with fat-tailed errors. We also provide some simulation results in Section 4 to assess the empirical performance of robust estimators. All proofs are relegated to the supplementary material.

We first introduce some notation. For a vector $\beta=(\beta_1,\dots,\beta_p)^\top$, let $|\beta|_1=\sum_{i}|\beta_i|$, $|\beta|_2=(\,{\sum_{i}\beta_i^2}\,)^{1/2}$, $|\beta|_0=|\{i:\beta_i\neq0\}|$ and  $|\beta|_\infty=\max_{i}|\beta_i|$. Let $\text{Supp}(\beta)$ be the support of $\beta$. For a matrix $A=(a_{ij})_{1\leq i,j\leq p}\in\mb{R}^{p\times p}$, let $\lambda_i$, $i=1,\dots,p$, be its eigenvalues and $\lambda_{\max}(A)=\max_{i}|\lambda_i|$ be the spectral radius, $\lambda_{\min}(A)=\min_{i}|\lambda_i|$. Let $\kappa(A)$ denote the condition number of $A$. Denote $|A|_1 = \sum_{i,j}|a_{ij}|$, $\|A\|_1 = \max_j \sum_i|a_{ij}|$, $\|A\|_\infty = \max_i\sum_{j} |a_{ij}|$, spectral norm $\|A\| = \|A\|_2 = \sup_{|x|_2 \neq 0}|Ax|_2/|x|_2$ and Frobenius norm $\|A\|_F = (\sum_{i,j} a_{ij}^2)^{1/2}$. Moreover, let $\text{tr}(A)$ be the trace of $A$, $\|A\|_{\max}=\max_{i,j}|a_{ij}|$ be the entry-wise maximum norm, $|A|$ be a matrix after taking absolute value of $A$, i.e. $|A|=(|a_{ij}|)_{i,j}$.  For a random variable $X$ and $q>0$, define $\|X\|_q=(\mE[|X|^q])^{1/q}$. For two real numbers $x,y$, set $x\lor y=\max(x,y)$. For two sequences of positive numbers $\{a_n\}$ and $\{b_n\}$, we write $a_n\lesssim b_n$ if there exists some constant $C>0$, such that $a_n/b_n\leq C$ as $n\to\infty$, and also write $a_n\asymp b_n$ if $a_n\lesssim b_n$ and $b_n\lesssim a_n$. We use $c_0,c_1,\dots$ and $C_0,C_1,\dots$ to denote some universal positive constants whose values may vary in different context. Throughout the paper, we consider the high dimensional regime, allowing the dimension $p$ to grow with the sample size $n$, that is, we assume $p=p_n\to\infty$ as $n\to\infty$.

\section{Bernstein-type Inequality for High Dimensional Linear Processes}
\label{sec: bernstein}

We consider a general framework of $p$-dimensional stationary linear process
\begin{equation}
\label{eq: linear process}
X_i = (X_{i1}, \ldots, X_{ip})^\top = \mu+ \sum_{k=0}^\infty A_k \boldsymbol{\varepsilon}_{i-k}
\end{equation}
where $\mu \in \mathbb{R}^p$ is the mean vector, $A_0=I_p$, $A_k$, $k\geq 1$, are $p \times p$ coefficient matrices with real entries such that $\sum_{k=0}^\infty \text{tr}(A_k^\top A_k)<\infty$,  $\boldsymbol{\varepsilon}_i=(\varepsilon_{i1}, \ldots, \varepsilon_{ip})^\top$, and $\varepsilon_{ij}$, $i\in \mathbb{Z}$, $1 \leq j \leq p$, are i.i.d.~random variables with zero mean and finite variance. Kolmogorov's three-series theorem ensures that the linear process (\ref{eq: linear process}) is well-defined. Many researchers have recently worked on this model, including \cite{bhattacharjee2014estimation, bhattacharjee2016large}, \citet{liu2015marvcenko}, and \cite{chen2016regularized}, among others. One special case of (\ref{eq: linear process}) is the stationary Gaussian process. If $A_k = 0$ for $k > d$, it becomes a vector moving average process of order $d$ (\cite{reinsel1997elements}, \cite{lutkepohl2005new}, \cite{brockwell2009time}). Another important class of models covered by (\ref{eq: linear process}) is the vector autoregressive (VAR) model, which has been widely used in economics and finance (e.g., \cite{sims1980macroeconomics}, \citet{stock2001vector}, \citet{tsay2005analysis}, \cite{fan2011sparse}).


The linear process (\ref{eq: linear process}) is a flexible multivariate model for correlated data in that the coefficient matrices $A_k$ capture both temporal and cross-sectional (spatial) dependence. Previous research has explored different structural conditions on the matrices $A_k$. For example, \cite{liu2015marvcenko} worked on a restrictive class of linear processes with matrices $A_k$ that are simultaneously diagonalizable, which implies the absence of spatial dependence among the components. \citet{bhattacharjee2016large} assumed that $\lim p^{-1}\tr(\Pi)$ exists and is finite for any polynomial $\Pi$ in $\{A_k, A_k^\top \}$, a joint convergence assumption that is difficult to verify. In this work, we shall impose a condition on the decay rate of the spectral norms of $A_k$, which allows for more general dependence structures and is easier to check in practice. Assume that there exist $0<\rho_p <1$ and $1\leq \gamma_p < \infty$ such that
\begin{equation}
\label{condition: gmc}
\|A_k\| = \sup_{|x|_2 \neq 0} \frac{|A_kx|_2}{|x|_2} \leq \gamma_p \cdot \rho_p^{k}
\end{equation}
for all $k \geq 0$. 
It implies short-range dependence in the sense that the autocovariance matrices $\text{Cov}(X_0, X_j) =\sum_{k=0}^\infty A_k A_{k+j}^\top$ are absolutely summable. 
The proposed quantities $\rho_p$ and $\gamma_p$ can capture temporal and spatial dependence of the underlying high-dimensional process. In particular, $\rho_p$ can depict the strength of temporal dependence, with smaller values indicating faster decay rates and weaker temporal dependence. And the magnitude of $\gamma_p$ can naturally quantify spatial dependence. A notable feature is that both $\gamma_p$ and $\rho_p$ may depend on $p$ in the high-dimensional regime. For example, when $p$ is large, $\rho_p$ may be a relatively large rate close to 1, indicating slow decay speed. In fact, there exists an absolute constant, independent of $p$ and strictly smaller than 1, such that (\ref{condition: gmc}) can be rephrased as
\begin{align}
\label{gmc revise}
\|A_k\| \leq \gamma_p \cdot \rho_0^{k/\tau_p} \text{ for some } \tau_p \geq 1.
\end{align}
Particularly, we define $\tau_p \equiv 1$ if there exists $\rho_0$ such that $\rho_p \leq \rho_0<1$, and $\tau_p = \log {\rho}_0/\log \rho_p$ for $\rho_0$ satisfying $0<\rho_0 \leq \rho_p$ if $\rho_p$ is large and increase with $p$. In the latter case, it could happen that $\tau:=\tau_p$ is an unbounded function in terms of the dimension $p$. It is worth noting that measures of dependence quantified by the dimension $p$ have been rarely explored in previous literature, despite their high relevance in analyzing high-dimensional time series. This feature is illustrated by the high dimensional vector autoregressive model in Example \ref{ex: VAR}. Thereafter, for notational simplicity, we omit the dimension subscript in $\gamma_p, \tau_p$, and refer them as $\gamma$, $\tau$. And we assume $\tau \leq n$; otherwise there may exist very strong temporal dependence in the sense that $\|A_k\|$ is decaying at a rate no faster than $\rho_0^{1/n}$. 

\begin{ex}[High Dimensional Vector Autoregressive Models]
\label{ex: VAR}
{\rm Consider the VAR(1) model  
\begin{equation}\label{var1}
X_i=AX_{i-1}+\boldsymbol{\varepsilon}_i,
\end{equation} 
where $A\in\mb{R}^{p\times p}$ is the transition matrix, and $\boldsymbol{\varepsilon}_i$, $i \in \mathbb{Z}$, are i.i.d.~error vectors with mean 0 and covariance matrix $I_p$. Equivalently it can be represented by the moving average model: $X_i  = \sum_{k=0}^\infty A^k \boldsymbol{\varepsilon}_{i-k}$, a special case of (\ref{eq: linear process}) with $A_k = A^k$. The process is stable (and hence stationary) if and only if the spectral radius $\lambda_{\max}(A)<1$ (\cite{lutkepohl2005new}). If $A$ is symmetric, as $\lambda_{\max}(A) = \|A\|$, condition (\ref{condition: gmc}) can be easily verified with $\rho_p = \lambda_{\max}(A)$ and $\gamma = 1$. For asymmetric $A$, it has a better interpretation when we look into condition (\ref{gmc revise}), and it could happen that $\tau$ may increase with the dimension $p$. Consider the design $A = (a_{ij})_{i,j=1}^p$ with $a_{ij} = \lambda^{j-i+1}{\bf 1}\{0 \leq j-i \leq B-1\}$ for some $0<\lambda<1$ and $1\leq B\leq p$. Here $B$ depicts how many variables at most in $X_{i-1}$ that have spatial effect on $X_{ij}$. Figure \ref{f1} delivers the plot of $\|A^k\|$ under the numerical setup $\lambda = 0.55$, $B=3,4$ and $p=20,25,30$. As can be seen, $\|A^k\|$ decays truly after a certain lag that is moving forward as $p$ is getting larger. This lag can be defined as $\tau$ in condition (\ref{gmc revise}), so $\tau$ increases with $p$ in this design. Additionally the geometric decay (its existence is to be shown later) occurs at a slow speed, viewed as another evidence of large $\rho_p$ (or large $\tau$ equivalently). For example, when $B=3$, $p=30$, $\|A^k\|$ roughly 
drops from 1.35 to 0.1 over a broad lag range from 30 to 60. The peak of $\|A^k\|$ before decay is defined as $\gamma$, indicating the strength of spatial dependence; by comparing the two plots, we can tell that stronger spatial dependence with a larger $B$ results to larger $\gamma$.}
\begin{figure}
    \centering
	\includegraphics[width=0.95\textwidth]{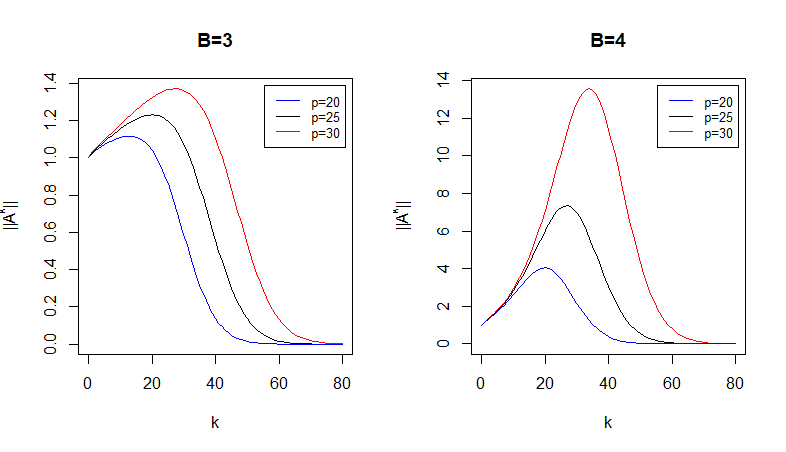}\\ \vspace{-0.7 cm}
	\caption{The graph of $\|A^k\|$ for $B=3, 4$ and $p=20,25,30$.}
	\label{f1}
\end{figure}
\end{ex}
Concentration inequalities play an important role in the study of sums of random variables. A number of inequalities have been derived for independent random variables; see \cite{buhlmann2011statistics} for a review. Bernstein's inequality (\cite{bernstein1946theory}) is one of the powerful tools when analyzing the concentration behavior by providing an exponential inequality for sums of independent random variables which are uniformly bounded. To fix the idea, let $Y_1,\dots,Y_n$ be i.i.d.~random variables such that $\mE Y_i=0$, $\text{Var}(Y_i)=\sigma^2<\infty$, and $|Y_i|\leq M$ for all $i$. Then for any $x>0$, one has
\begin{equation}\label{bernstein1946}
	\mP\Big(\sum_{i=1}^nY_i\geq x\Big)\leq\exp\Big\{-\frac{x^2}{2n\sigma^2+2Mx/3}\Big\},
\end{equation}
which suggests two types of bound for tail probability: sub-Gaussian-type tail $\exp\{-x^2/(Cn\sigma^2)\}$ in terms of the variance of $\sum_{i=1}^nY_i$ and sub-exponential-type tail $\exp\{-x/(CM)\}$ in terms of the uniform bound $M$. Bernstein type inequalities have been developed for Markov chains or temporally dependent processes with an additional order ($\log n$ in some constant powers) in the sub-exponential-type tail; see, for example, \cite{adamczak2008tail}, \cite{merlevede2009bernstein}, \cite{hang2017bernstein} and \cite{zhang2019robust}. The problem of generalizing to high dimensional time series is quite challenging and very few results have been obtained. Our first goal is to establish a new Bernstein type inequality for the sum of a bounded transformation of the high dimensional linear processes in (\ref{eq: linear process}).

\begin{thm}\label{thmbern}
Let $X_i$ be the linear process generated from (\ref{eq: linear process}) with $\mE \varepsilon_{ij}=0$, $\mE \varepsilon_{ij}^2 =\sigma^2 <\infty$ and condition (\ref{gmc revise}) be satisfied. Let $G:\mb{R}^p\to\mb{R}$ be a function with $|G(u)|\leq M$ for all $u\in\mb{R}^p$. Suppose there exists a vector $g = (g_1, \ldots, g_p)^\top$ with $g_i\geq0$ and $\sum_{i=1}^p g_i=1$ such that the following Lipschitz condition holds: for all $u = (u_1, \ldots, u_p)^\top$ and $v=(v_1, \ldots, v_p)^\top$,
		 \begin{equation}\label{lips}
		 	|G(u)-G(v)|\leq \sum_{i=1}^p g_i|u_i-v_i|.
		 \end{equation}
		 Then for any $x>0$, we have
		 \begin{align}\label{thmb}\mP\Big(\sum_{i=1}^nG(X_i)-\mE G(X_i)\geq x\Big)&\leq 2\exp\bigg\{-\frac{x^2}{C_1 n\sigma^2\tau^2\gamma^2+C_2\tau Mx}\bigg\},\end{align}
	where the constants $C_1$ and $C_2$ are given by
	\begin{align*}C_1=\frac{16\me^2}{\sqrt {2\pi}\rho_0^4[\log(1/\rho_0)]^3},\quad
	 C_2&=\frac{8\me }{\log(1/\rho_0)}.\end{align*}
\end{thm}

\begin{rem}
Equipped with our new inequality \eqref{thmb}, one can investigate the concentration properties of sums of bounded transformations of high dimensional linear processes that exhibit both temporal and cross-sectional dependence, characterized by $\tau$ and $\gamma$ respectively. In the special case where the processes are one-dimensional, denoted by $X_i\in \mathbb{R}$, and $\tau=1$ and $\gamma$ is of a constant order that satisfies condition (\ref{condition: gmc}), our probability inequality \eqref{thmb} is just as sharp as the classical Bernstein's inequality \eqref{bernstein1946}. It is worth mentioning that our inequality is strictly sharper than the existing Bernstein-type inequalities for univariate time series established by \cite{merlevede2009bernstein} and \cite{zhang2019robust}. To fix the idea, we shall recall that \cite{merlevede2009bernstein} derived a concentration inequality for univariate strong mixing process $(X_i)$ of mean 0 and upper bounded by $M$ in magnitude:
	\begin{align}\label{m2009}
		\mP\Big(\sum_{i=1}^n X_i\geq x\Big)\leq\exp\Big\{-\frac{Cx^2}{nv^2+M^2+M(\log n)^2x}\Big\},
	\end{align}
	where $v^2$ is the asymptotic variance of $\sum_{i=1}^nX_i/\sqrt n$. \cite{zhang2019robust} obtained a similar bound with $v^2$ represented in terms of functional dependence measures. In our framework of linear processes with condition (\ref{condition: gmc}) satisfied, $v^2 \asymp \sigma^2 \gamma^2$ can be computed for one-dimensional cases. Notably, our inequality is sharper by removing the additional factor $(\log n)^2$ in the sub-exponential type bound.
\end{rem}

To study high dimensional time series, an important class of transformations is linear combinations of transformed component processes, that is, $G(X_i) = \sum_{j=1}^p a_j h_j(X_{ij})$, where $\sum_{j=1}^n |a_j|=1$, $h_j: \mathbb{R} \to \mathbb{R}$ are univariate functions satisfying $|h_j(x)|\leq M$ and $|h_j(x)-h_j(y)|\leq 1$ for any $x, y \in \mathbb{R}$, thus condition \eqref{lips} is satisfied with $g_j=|a_j|$. As a special case, when $G(X_i) = h_j(X_{ij})$, for a fixed $1\leq j \leq p$, the result provides a concentration inequality for sums of each component process $(X_{ij})_{i \in \mathbb{Z}}$ after the transformation $h_j$. This is useful in the application of estimating the mean vector of high-dimensional linear processes in a robust way, as discussed at the end of this section. In Remark 2.3, we highlight that our inequality yields a rate of $\ell_\infty$ norm convergence for the robust mean estimator that is as sharp as the optimal rate for i.i.d.~processes.


Condition (\ref{gmc revise}) requires $\|A_k\|$ geometrically decayed  up to the quantity $\gamma$ and the decay speed is controlled by $\tau$. \citet{chen2016regularized} worked on the same linear model under a weaker condition allowing polynomial decay, namely, $\|A_k\|=O((1 \vee k)^{-\alpha})$ for some $\alpha>1$, under which, it is noteworthy that an exponential type probability inequality does not hold in general even if it is one dimensional process with a uniform bound. 
That is to say, if we relax the condition (\ref{condition: gmc}) to a polynomial decay, the concentration inequality delivers an exact rate with algebraic decay for one dimensional linear process; see Theorem 14 in \citet{chen2018concentration}.

In Theorem \ref{thmbern}, the existence of a finite variance of $\varepsilon_{ij}$ is assumed. If it is relaxed to the existence of finite exponential moment, a similar bound can be achieved with $G$ not necessarily bounded; see Theorem \ref{thmbern2} below. 

\begin{thm}
\label{thmbern2}
In the model (\ref{eq: linear process}), assume that $\mE \varepsilon_{ij}=0$, $\mE \exp(c_0|\varepsilon_{ij}|) = \theta <\infty$ for some constant $c_0>0$ and condition (\ref{gmc revise}) is met. Then for $G$ satisfying (\ref{lips}), it holds that
		 \begin{align}\label{bern2}
		 \mP\Big(\sum_{i=1}^nG(X_i)-\mE G(X_i)\geq x\Big)&\leq 2\exp\bigg\{-\frac{x^2}{C_3 n\theta^2\tau^2\gamma^2+ C_4 \gamma \tau x }\bigg\},\end{align}
	where the constants $C_3$ and $C_4$ depend on $\rho_0$ and $c_0$.
\end{thm}

One immediate application of Theorem \ref{thmbern} is to estimate the mean vector for  high dimensional fat-tailed linear processes. From an $M$-estimation viewpoint, we apply Huber's estimator (\cite{huber1964robust}) of the mean vector, denoted by $\hat{\mu}=(\hat{\mu}_1,\ldots, \hat{\mu}_p)^\top$, with $\hat{\mu}_j$ as the solution of $a$ to the equation 
\begin{equation*}\sum_{i=1}^n \phi_{\nu}(X_{ij}-a) = 0,
\end{equation*}
where $\phi_{\nu}(x)=(x\wedge \nu)\vee(-\nu)$ is the Huber function with the robustification parameter $\nu > 0$. 

\begin{thm}
\label{thm: mean}
Let $X_i$ be generated from model (\ref{eq: linear process}) with $\mE \varepsilon_{ij}=0$, $\text{Var}(\varepsilon_{ij})=1$, $\mu = \mE X_{i}$ and $\max_{1\leq j\leq p}\text{Var}(X_{ij})=\mu_2^2 <\infty$. 
Choose $\nu\asymp \mu_2\sqrt{{n}/{\log p}}$.
With probability at least $1-4p^{-c}$ for some $c>0$, it holds that
\begin{equation}
\label{eq: mean estimation}
|\hat{\mu}-\mu|_\infty \leq C(\gamma + \mu_2)\tau\sqrt{\frac{\log p}{n}}
\end{equation}
under the scaling condition $(\gamma+\mu_2)\tau \sqrt{\log p/n} \to 0$, where $C$ is a positive constant depending on $c$ and the constants $C_1, C_2$ in Theorem \ref{thmbern}.
\end{thm}

\begin{rem}
Theorem \ref{thm: mean} delivers the rate of $\ell_\infty$ norm convergence for the robust mean estimator $\hat{\mu}$ and it involves a delicate interplay with the cross-sectional dependence, temporal dependence and the variance of the process. If $\gamma, \mu_2$ and $\tau$ are all of a constant order, it follows that 
\begin{equation}
\label{rate mean}
|\hat{\mu} - \mu|_\infty = O_{\mP}(\sqrt{\log p / n}),
\end{equation}
under the scaling condition $\log p/n \to 0$. We shall remark that (\ref{rate mean}) is as sharp as the optimal rate provided in Theorem 5 of \citet{fan2017estimation} concerning the concentration of the mean estimation for the i.i.d.~case. And it is strictly sharper than the results using existing Bernstein type inequalities for time series such as the ones in \cite{merlevede2009bernstein}, \cite{hang2017bernstein} and \cite{zhang2019robust}.
\end{rem}

\section{Robust Estimation of Time Series Regression}
\label{sec: robust estimation}
In this section, we shall investigate robust estimation of high dimensional time series linear regression and autoregression with fat-tailed covariates and errors. It is expected that our framework of high dimensional linear processes and these Bernstein type inequalities will be useful in other high-dimensional estimation and inference problems that involve dependent and non-sub-Gaussian random variables.

\subsection{Estimating Time Series Regression with Correlated Errors}
\label{sec: linear regression}
We work on linear regression models with random design that involve time dependent covariates and errors:
\begin{equation}
\label{label: linear}
Y_i = X_i^\top \beta^* + \xi_i,
\end{equation}
with more justification provided as follows.

\textbf{Assumptions.}
\begin{itemize}
\item[(A1)] $X_i$ is generated from the  $p$-dimensional linear process $X_i = \sum_{k=0}^\infty A_k\boldsymbol{\varepsilon}_{i-k}$ where the components of $\boldsymbol{\varepsilon}_i$ are i.i.d.~random variables with $\mE(\varepsilon_{ij})=0$  and $\text{Var}(\varepsilon_{ij})=\sigma_{\varepsilon}^2 <\infty$. Condition (\ref{gmc revise}) is satisfied with $\gamma$ and $\tau$, which may depend on $p$.
\item[(A2)] $\xi_i = \sum_{k=0}^\infty b_k \eta_{i-k}$ is the error process, where $\eta_{i}$ are i.i.d.~random variables with $\mE(\eta_{i})=0$ and $\text{Var}(\eta_{i})=\sigma_{\eta}^2<\infty$, and $b_k \leq C\rho^{k}$ for universal constants $0<\rho<1$ and $C<\infty$. 
\item[(A3)] $X_i$ is strictly exogenous in the sense that $(\boldsymbol{\varepsilon}_i)_i$ are independent of $(\eta_i)_i$, where $(\boldsymbol{\varepsilon}_i)_i$ and $(\eta_i)_i$ are error processes of $X_i$ and $\xi_i$ respectively as defined in (A1) and (A2). 
\end{itemize}
The framework (\ref{label: linear}) is quite general as the linear process includes a wide range of commonly used time series models. For linear regression models with dependent errors, earlier work mainly dealt with fixed design or i.i.d.~covariates. \cite{wang2007regression} and \cite{yoon2013penalized} considered the case where $\xi_i$ follows an autoregressive process, one type of linear processes.  \cite{gupta2012note} concerned weakly dependent $\xi_i$ introduced by \cite{doukhan1999new} and specifically discussed the AR(1) and ARMA cases. \cite{alfons2013sparse} adopted the same format of moving average errors but assumed long memory dependence. More generally, \cite{wu2016performance} and \cite{chernozhukov2021lasso} considered the nonlinear Wold representation with $X_i = g(\ldots, \boldsymbol{\varepsilon}_{i-1}, \boldsymbol{\varepsilon}_{i})$ and $\xi_i = h(\ldots, \eta_{i-1}, \eta_{i})$. 

We form a modified $\ell_1$-regularized Huber estimator of $\beta$, given by
\begin{equation*}
\hat{\beta} = \text{arg min}_{\beta\in\mb{R}^{p}} \frac{1}{n}\sum_{i=1}^n \Phi_{\nu}((Y_i-X_i^\top \beta)w(X_i)) + \lambda |\beta|_1,
\end{equation*}
where  $\Phi_{\nu}$ is Huber loss function (\cite{huber1964robust}) 
\begin{equation*}
\Phi_{\nu} (x) 
=\left\{
\begin{aligned}
&x^2/2, &\text{if } |x| \leq \nu,\\
&\nu|x|-\nu^2/2, &\text{if } |x| > \nu,
\end{aligned}
\right.
\end{equation*}
defined with respect to the robustification parameter $\nu >0$. More properties of Huber regression are referred to \cite{huber1973robust}, \cite{yohai1979asymptotic}, 
\cite{mammen1989asymptotics}, \cite{sun2020adaptive}, \cite{fan2017estimation} among others. Motivated by \cite{loh2021scale}, $w(x): \mathbb{R}^p \to \mathbb{R}$ is a weight function defined by $$w(x) = \min\Big\{1, \frac{b}{|Bx|_2}\Big\}$$ 
where $b \in \mathbb{R}$ is a fixed constant and $B\in \mathbb{R}^{p\times p}$ is a provided positive definite matrix. 
With such a choice of $w(x)$, it always holds that $|w(x)x|_2\leq b/\lambda_{\min}(B)=:b_0$. Different from the regular Huber regression concerning well-behaved $X_i$ (e.g., Gaussian or sub-Gaussian), an additional weight function is incorporated on the covariate process to account for the fat tails of $X_i$. In Section S1, we conduct a simulation study for robust time series regression estimation and look into the effect of $w(x)$. 

As a popular convention, $\beta^*$ is assumed to be sparse in the sense that $|\beta^*|_0=s$. Denote the condition number of $B$ as $\kappa(B)=\lambda_{\max}(B)/\lambda_{\min}(B)$. Theorem \ref{thm: linear reg} below concerns the estimation consistency of $\hat{\beta}$. 

\begin{thm}
\label{thm: linear reg} Let Assumptions (A1) (A2) (A3) be satisfied. Assume
\begin{equation}
\label{scaling condition}
b_0 (b_0+\kappa(B)\gamma\sigma_{\varepsilon})\tau\sqrt s\sqrt{(\log p)^3/n} \to 0.
\end{equation}
Choose $\nu \asymp \sigma_{\eta}(n/\log p)^{1/2}$ and $\lambda  \asymp b_0 \sigma_{\eta} (\log p/n)^{1/2}$. 
With probability at least $1-8p^{-c}$ for some $c>0$,  it holds that that
\begin{equation}
\label{eq: rate}|\widehat{\beta}-\beta|_2 \leq C \frac{b_0\sigma_{\eta}}{\lambda_{\min}(\mE[\frac{w^2(X_i)}2X_iX_i^\top])}\sqrt{\frac{s\log p}{n}}.
\end{equation}
\end{thm}


The scaling condition (\ref{scaling condition}) to ensure consistency indicates a subtle interplay with the dimensionality parameters ($s, p, n$),  internal parameters ($\tau, \gamma, \sigma_{\varepsilon}$), and the known values $b_0$ and $\kappa(B)$ associated with the weight function $w(x)$. 
The convergence rate (\ref{eq: rate}) scales inversely with the quantity $\lambda_{\min}(\mE[\frac{w^2(X_i)}2X_iX_i^\top])$ and it suggests that we can not shrink the covariates too aggressively. If $X_i$ is well-behaved with the existence of finite exponential moment, one may eliminate the weight function and replace the factor by the larger quantity $\lambda_{\min}(\mE[X_iX_i^\top])$. 

In the extensively studied regression setting with i.i.d.~covariates, \cite{fan2017estimation} delivered an optimal convergence rate of $|\widehat{\beta}-\beta|_2$ for weakly sparse model under the fat tails (the same as the minimax rate in \cite{raskutti2011minimax}). In the special exact sparse case, their convergence rate is $\sqrt{s(\log p)/n}$ and it relies on the sub-Gaussian tail assumption for the covariates $X_i$. \cite{loh2021scale} allowed broader classes of distributions for $X_i$ by inserting a weight function to control the Euclidean norm of $X_i$, but required the errors drawn i.i.d.~from a symmetric distribution and thus selected $\nu$ at a fixed constant order (cf.~Theorem 1), while \cite{fan2017estimation} waived the symmetry requirement by allowing $\nu$ to diverge in order to reduce the bias induced by the Huber loss when the distribution of $\xi_i$ is asymmetric. We borrow the ideas from both and further account for time dependent covariates and errors. Compared to \cite{loh2021scale} with i.i.d.~covariates and i.i.d.~errors, our result requires a stronger scaling condition (\ref{scaling condition}) in terms of the dependence quantities $\gamma, \tau$ and a larger power of $\log p$, by concerning both dependent covariates and errors. 

Applying $\ell_1$ regularization in time series regression, \cite{wu2016performance} (cf.~Theorem 5) dealt with correlated covariates and errors and allowed a wider class of stationary processes in a causal form. The linear error process in our consideration falls in the weaker dependence range within their framework. If $\gamma, \tau, \sigma_{\eta} = O(1)$, $p=o(n^{q-1})$ is required for their regular estimator without accounting for robustness, where $q>2$ is the order of finite moments for $\xi_i$. \cite{chernozhukov2021lasso} considered the Lasso estimator for a system of time series regression equations with one regression equation as a special case, for which the allowed dimension is still of a polynomial rate to ensure consistency by looking into the performance bound with respect to the prediction norm (cf.~Corollary 5.4). In comparison with the above two work, we can allow a much wider range for the dimension $p$ under mild conditions.

The tuning parameter $\nu$ plays a key role by adapting to errors with fat tails. In practical applications, the optimal values of the
tuning parameters $\nu$ and $\lambda$ can be chosen by a two-dimensional grid search using cross-validation or information-based criterion such as AIC or BIC. We leave theoretical investigation on selecting the tuning parameters as important future work.

\subsection{Estimating Transition Matrix in VAR Models}
\label{sec: VAR estimation}
To study the evolution of a set of endogenous variables over time, a popular choice is vector autoregression. Interpretations of large vector autoregressive models have been developed in various applications such as policy analysis (\cite{sims1992interpreting}), financial systemic risk analysis (\cite{Gourieroux2011}), portfolio selection (\citet{ledoit2003improved}), functional genomics (\cite{shojaie2012adaptive}) and brain networks (\cite{sameshima2014methods}). 

As a general VAR model of order $d$ can be reformulated as a VAR(1) model by appropriately redefining the random vectors, much work (\cite{han2015direct}, \cite{guo2016high}) considered the model with lag 1 as given in \eqref{var1}. Among the work concerning high dimensional vector autoregressive models, most investigations require the Gaussian assumption (\cite{kock2015oracle}, \cite{basu2015regularized}, \cite{han2015direct}) or some structure assumption stronger than the minimal requirement $\lambda_{\max}(A)<1$; for example, \cite{han2015direct} imposed $\|A\|<1$, and \cite{guo2016high} considered banded $A$ with some decay condition on $\|A^k\|$ free of $p$. For many VAR designs (Example \ref{ex: VAR} is one such), it could happen that $\|A\|\geq 1$ and the dimension $p$, as the size of $A$, can play a role in measuring the temporal and cross-sectional dependence. \cite{basu2015regularized} proposed stability measures to capture temporal and cross-sectional dependence. From a different viewpoint, we try to fill in the gap between the spectral radius of a matrix and its spectral norm. Intuition can be gained from the proposition below. It provides a sufficient and necessary condition for $\lambda_{\max}(A)<1$ by relating to the spectral norm. 
\begin{prop}\label{p1}
For any matrix $A$, it holds that $\lambda_{\max}(A) < 1$ if and only if there exists some finite integer $k$ such that $\|A^k\| \leq \rho_0$ given any universal constant $0 < \rho_0 < 1$. 
\end{prop}
Letting $\tau =\min\{k\in\mb{Z}^+: \|A^k\| \leq \rho_0\}$ and $\gamma = \rho_0^{-1} \max_{0\leq k \leq \tau-1}\|A^k\|$, condition (\ref{gmc revise}) holds for the model (\ref{var1}) without extra requirement.
We now introduce the notation. Let ${\boldsymbol a}_{j\cdot}^\top$  be the $j$-th row of $A$ and
$s_j$ be the cardinality of the support set of ${\boldsymbol a}_{j\cdot}$, i.e., $s_j=|\text{supp}({\boldsymbol a}_{j\cdot})|=|\{i: a_{ij}\neq 0\}|$. Denote $s = \max_{1\leq j \leq p}s_j$ and $\mathcal{S}=\sum_{i=j}^ps_j$. 
For robustness, we first truncate the data by obtaining $\tilde{X_i} = \phi_{\nu}(X_i)$, where $\nu$ is the truncation parameter and is to be determined later. 
For notational convenience, we assume $X_0$ is also observed. Based on the truncated sample $\widetilde X_i$ and tuning parameter $\lambda>0$, we propose to estimate $A$ by solving the following Lasso problem:
\begin{equation}
\label{eq: lasso}
	\widehat{A}=\text{arg min}_{B\in\mb{R}^{p\times p}} \frac1n\sum_{i=1}^n |\widetilde{X}_{i} - B\widetilde{X}_{i-1}|_2^2+\lambda|B|_1,
\end{equation}
which is equivalent to solving the $p$ sub-problems:
\begin{equation}
\label{subproblem for Lasso}
	\widehat{\boldsymbol a}_{j\cdot}=\text{arg min}_{{\boldsymbol b}\in\mb{R}^{p}}\frac1n\sum_{i=1}^n(\widetilde{X}_{ij}-{\boldsymbol b}^\top\widetilde{X}_{i-1})^2+\lambda|\boldsymbol{b}|_1.
\end{equation}
Before proceeding, we state the key assumptions on the process (\ref{var1}) and some scaling conditions to guarantee consistency of the robust estimator $\hat{A}$. 

\textbf{Assumptions.}\\
\noindent (B1) $\mE \varepsilon_{ij} = 0$; $\mE\varepsilon_{ij}^2 = 1$; $\max_{1\leq j \leq p} \|X_{ij}\|_q = \mu_q<\infty$ for some $q>2$.\\
\noindent (B2) $\mu_q \gamma\tau s^2[(\log p)/{n}]^{(q-2)/(2q-2)}\to0.$  \\
\noindent (B2') $\mu_q \gamma\tau {\mathcal{S}}^2[(\log p)/{n}]^{(q-2)/(2q-2)}\to0.$

Assumption (B1) imposes polynomial moment conditions for the underlying VAR process. 
Assumption (B2) (or (B2')) assumes a vanishing scaling property. If $\mu_q$, $\tau$ and $\gamma$ are of a constant order, (B2) is reduced to the scaling condition that involves $s$ (or $\mathcal{S}$), $n$ and $p$ only. 
\begin{thm}\label{thmestimation} Let Assumptions (B1) and (B2) be satisfied. Choose the truncation parameter $\nu\asymp \mu_q(n/\log p)^{1/(2q-2)}$. Let $\widehat{A}$ be the solution of (\ref{eq: lasso}) with $\lambda\asymp \mu_q\gamma\tau(\|A\|_{\infty}+1)[(\log p)/n]^{(q-2)/(2q-2)}$.
	It holds that
	\begin{eqnarray}\label{eq: thmestimation}
	\|\widehat{A}-A\|_{\infty} \leq C \mu_q\gamma\tau (\|A\|_{\infty}+1)s\bigg(\frac{\log p}{n}\bigg)^{\frac{1}{2}-\frac{1}{2q-2}}
   \end{eqnarray}
	with probability at least $1-8p^{-c}$ for some constant $c>0$. If Assumption (B2') is further satisfied, it also holds that	\begin{equation}\label{Frobenous norm}
	 \|\widehat{A}-A\|_{F}  \leq C'  \mu_q\gamma\tau(\|A\|_{\infty}+1)\sqrt{\mathcal{S}}\left(\frac{\log p}n\right)^{\frac{1}{2}-\frac{1}{2q-2}}
	\end{equation}
	with probability at least $1-8p^{-c}$ for some constant $c>0$.
\end{thm}

The obtained rates of convergence are governed by two sets of parameters: (i) dimensionality parameters: the dimension $p$, 
sparseness parameter $s$ (or $\mathcal{S}$), and the sample size $n$; (ii) internal parameters: the moment $\mu_q$, dependence quantities $\tau$, $\gamma$, and the maximum absolute row sum $\|A\|_\infty$.
If the internal parameters are assumed to be of a constant order, we can get
\begin{equation*}
\|\widehat{A}-A\|_F = O_{\mathbb{P}}\Big(\sqrt{\mathcal{S}}\big(\frac{\log p}{n}\big)^{\frac{1}{2}-\frac{1}{2q-2}}\Big).
\end{equation*}
To ensure consistency, the dimension $p$ can be allowed to increase exponentially with $n$ in view of the mild scaling condition. Compared to \cite{guo2016high} with the same constant order of internal parameters, they can only allow the narrower range $p=o(n^c)$ for some $0<c<(q-4)/8$ (cf.~Condition 4(i)). For Gaussian autoregressive models,  proposition 4.1 of \cite{basu2015regularized} suggests the order in terms of dimension parameters as
\begin{equation*}
\|\widehat{A}-A\|_F = O_{\mathbb{P}}\Big( \sqrt{\mathcal{S}}\sqrt{\frac{\log p}{n}}\Big).
\end{equation*}
In the presence of fat tails with the existence of finite $q$-th moment, our result yields a slightly slower convergence rate characterized by the moment order $q$ and it is closer to their bound when $q$ gets larger.

As an alternative method, the idea of Dantzig-type estimation (\cite{candes2007dantzig}, \cite{cai2011constrained}, \cite{han2015direct}) can be modified in the robust way. 
Let $\Sigma_k$ denote the autocovariance matrix of the process $(X_i)$ at lag $k$. The celebrated Yule-Walker equation $A=\Sigma_0^{-1}\Sigma_1$ suggests that a good estimate $\widehat A$ should have a small error in terms of $\|\Sigma_0\widehat A-\Sigma_1\|_{\max}$. Without direct access to the autocovariance matrices $\Sigma_0$ and $\Sigma_1$, a natural approach is to find nice estimators for them. \cite{han2015direct} used sample autocovariance matrices and enjoyed a nice performance bound under Gaussianity. For fat-tailed cases, we 
consider the robust estimators of the autocovariance matrices based on the truncated sample:
$$\widehat\Sigma_k=\frac{1}{n}\sum_{i=1}^{n}\widetilde X_{i-k}\widetilde X_{i}^\top, \text{ for  }k=0,1.$$
The Dantzig- type estimator is then modified to solving the following convex programming:
\begin{align}\label{Dantzig}\widehat A=\text{arg min}_{B\in\mb{R}^{p\times p}} |B|_1 \ \ \text{s.t. } \ \ \|\widehat\Sigma_0 B-\widehat\Sigma_1\|_{\max}\leq\lambda,
\end{align}
where $\lambda>0$ is a tuning parameter. 
Observe that problem \eqref{Dantzig} can be solved in parallel, namely, \eqref{Dantzig} is equivalent to $p$ subproblems:
\begin{align}\label{thetahat}\widehat {\boldsymbol a}_{\cdot j}=\text{arg min}_{{\boldsymbol b}\in\mb{R}^{p}} |{\boldsymbol b}|_1 \ \ \text{s.t. } \ \ |\widehat\Sigma_0 {\boldsymbol b}-\widehat{\Sigma}_1 u_j|_{\infty}\leq\lambda,\quad j=1,\dots, p
\end{align}
for any unit vector $u_j$. Let ${\boldsymbol a}_{\cdot1},{\boldsymbol a}_{\cdot2},\dots, {\boldsymbol a}_{\cdot p}$ be columns of $A$ and denote $s^* = \max_{1\leq j \leq p}|\text{supp}({\boldsymbol a}_{\cdot j})|$. 
We can obtain $\widehat A$ by simply concatenating all the columns $\widehat{\boldsymbol a}_{\cdot j}$, i.e. $\widehat A=(\widehat{\boldsymbol a}_{\cdot 1}, \widehat{\boldsymbol a}_{\cdot 2}, \dots, \widehat{\boldsymbol a}_{\cdot p})$. The next theorem delivers an upper bound on the statistical accuracy.

\begin{thm}\label{thmdan}
Let Assumption (B1) be satisfied. Let $\widehat{A}$ be the solution of (\ref{Dantzig}) with $\nu\asymp \mu_q(n/\log p)^{1/(2q-2)}$ and $\lambda \asymp \mu_q \gamma\tau(\|A\|_1+1)[(\log p)/n]^{(q-2)/(2q-2)}$. With probability at least $1-8p^{-c'}$ for some constant $c'>0$, it holds that
\begin{eqnarray}\label{eq: max norm}
&&	\|\widehat A-A\|_{\max}\leq C \mu_q\gamma\tau\|\Sigma_0^{-1}\|_1(\|A\|_1+1)\bigg(\frac{\log p}n\bigg)^{\frac{1}{2}-\frac{1}{2q-2}}, \\\label{thmeqdan}
 &&\|\widehat A-A\|_1\leq C' \mu_q\gamma\tau\|\Sigma_0^{-1}\|_1(\|A\|_1+1)s^*\bigg(\frac{\log p}n\bigg)^{\frac{1}{2}-\frac{1}{2q-2}}.
\end{eqnarray}
\end{thm}

It is interesting to see that the convergence rate of the modified Dantzig-type estimator has a similar form to that of the robust Lasso estimator developed in Theorem \ref{thmestimation}, if the included internal parameters for the process are of a constant order. Both methods involve $p$ parallel programming problems with the lasso-based one concerning row-by-row estimation while the Dantzig method concerning column-by-column instead. The situation $\|A\|<1$ studied by \cite{han2015direct} is the special case where $\gamma=1$ and $\tau =1$ in our framework. In their paper, a more flexible sparse condition was imposed: the transition matrix $A$ belongs to a class of weakly sparse matrices defined in terms of strong $\ell^r$-ball ($0\leq r <1$), which was also considered by \cite{bickel2008covariance}, \cite{rothman2009generalized},
\cite{cai2011constrained}, \cite{cai2012optimal} in estimating covariance and precision matrices.  For $r=0$, it is the exact sparse case and Theorem 1 in \cite{han2015direct} implies the dimension parameter order
\begin{equation*}
\|\widehat{A}-A\|_1 = O_{\mathbb{P}}\Big(s^*\sqrt{\frac{\log p}{n}}\Big),
\end{equation*}
a bit sharper than our result (\ref{thmeqdan}). There is additional cost for fat-tailed processes with robustness absorbed. We shall remark that we can also derive the bound of $\|\widehat{A}-A\|_1$ accordingly for weakly sparse $A$ based on the result (\ref{eq: max norm}) without any technical difficulty.

\section{Simulation Study}
In this section, we evaluate the finite sample performance of both robust Lasso and Dantzig estimators that are proposed in Section \ref{sec: VAR estimation} and compare with the traditional Lasso and Dantzig methods. Simulation on time series linear regression is presented in the supplementary material. We consider the model \eqref{var1}, where $\varepsilon_{ij}$'s are i.i.d.~standardized Student's $t$-distributions with $\text{df}=5$. We take the numerical setup of $n=50,100$ and $p=50, 100, 500$ and set $s =\lfloor \log p \rfloor $. For the true transition matrix $A=(a_{ij})$, we consider the following designs.
\begin{itemize}
	\item [(1)] Banded: $A = (\lambda^{|i-j|}{\bf 1}\{|i-j|\leq s\})$ and $\lambda = 0.5$. 
	\item [(2)] Block diagonal: $A = \text{diag}\{A_i\}$, where each $A_i \in \mathbb{R}^{s\times s}$ follows the structure in Example \ref{ex: VAR} with $B=2$ and $\lambda_i \sim Unif(-0.8, 0.8).$
	\item[(3)] Toeplitz: $A = (\lambda^{|i-j|})$ and $\lambda = 0.5$.
	\item[(4)] Random Sparse: $a_{ii} \sim Unif(-0.8,0.8)$ and $a_{ij}\sim N(0,1)$ for $(i,j)\in C \subset \{(i,j): i\neq j\}$ where $C$ is randomly selected and $|C| = s^2$.
\end{itemize}
To ensure stationarity of the VAR model, the designs in (1), (3), and (4) were further scaled by a factor of $2\lambda_{\max}(A)$ to ensure that the spectral radius of the transition matrix is less than 1. Figure \ref{f2} shows the plot of $\|A^k\|$ under the four designs with $p=100, 500$. These patterns of matrix $A$ were previously studied in \cite{han2015direct}, where the assumption $\|A\|<1$ was necessary. In this study, we keep the designs of symmetric sparse and weakly sparse matrices, which are presented in cases (1) and (3), respectively. For these two cases, it holds that $\|A^k\| = (\lambda_{\max}(A))^k = (0.5)^k$, and condition (\ref{gmc revise}) is satisfied with $\tau=1$, $\gamma = 1$ and $\rho_0=0.5$. However, for the designs using asymmetric coefficient matrices (cases (2) and (4)), we allow $\|A\|>1$, and $\tau$ and $\gamma$ in condition (\ref{gmc revise}) may depend on the value of $p$.

\begin{figure}
    \centering
	\includegraphics[width=0.9\textwidth]{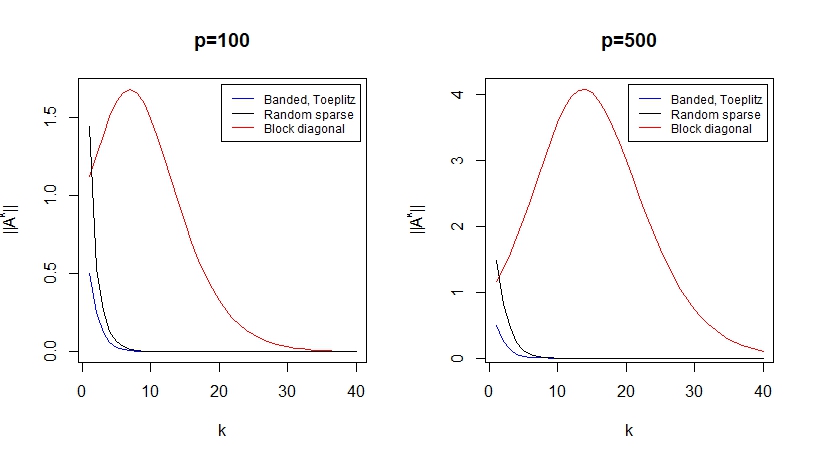} \\ \vspace{-0.7 cm}
	\caption{The graph of $\|A^k\|$ for the four designs of $A$ with $p=100, 500$}
	\label{f2}
\end{figure}


In each repetition, we generate a process of length $2n$. We run the estimation procedure in (\ref{eq: lasso}) or (\ref{Dantzig}) based on $\{X_{1},\ldots, X_n\}$ by a two-dimensional grid search for the tuning parameters $\nu$ and $\lambda$. For each $(\nu, \lambda)$ in the grid, denote the estimator by $\hat{A}(\nu, \lambda)$. Then $(\nu, \lambda)$ is chosen such that $n^{-1}\sum_{t=n+1}^{2n}|X_t - \hat{A}(\nu, \lambda)X_{t-1}|_2^2$, the average prediction error on $\{X_{n+1}, \ldots, X_{2n}\}$, is minimized.
The following tables report the average (as
an entry) and standard deviation (in parentheses) of estimation error based on 1000 repetitions in different matrix norms consistent with Theorem \ref{thmestimation} and Theorem \ref{thmdan}. As comparisons, we obtain the results for robust methods and the traditional versions (Lasso estimator in \cite{tibshirani1996regression} and Dantzig-based estimator in \cite{han2015direct}) in different designs.

\begin{table}
\begin{center}
\small
\begin{tabular}{|c|c|c|c|c|}
 \hline
$p=50, n=100$ &  Banded & Block & Toeplitz & Random \\ 
\hline 
Lasso $L_\infty$  & 1.49 (0.060) & 0.96 (0.072) & 1.46 (0.143) & 1.28 (0.124)\\ 
 
Lasso  $L_F$ & 1.56 (0.112) & 1.22 (0.112) & 1.55 (0.121) & 1.30  (0.084)\\ 

Robust-Lasso  $L_\infty$ & 1.35 (0.049) & 0.80 (0.078)& 1.30 (0.072) & 1.17 (0.065)\\ 

Robust-Lasso $L_F$  & 1.37 (0.076) & 1.05 (0.041)& 1.36 (0.090)& 1.23 (0.038)\\ 
 \hline
 \hline
Dantzig $L_1$  & 2.01 (0.121) & 1.91 (0.087) & 2.02 (0.140) & 2.40 (0.159)\\ 
 
Dantzig  $L_F$ & 2.10 (0.095) & 1.92 (0.074) & 2.04 (0.125) & 2.69  (0.078)\\ 

Robust-Dantzig  $L_1$ & 1.86 (0.050) & 1.08 (0.058)& 1.86 (0.043) & 1.47 (0.077)\\ 

Robust-Dantzig $L_F$  & 1.90 (0.049) & 1.41 (0.044)& 1.89 (0.033)& 2.02 (0.073)\\ 
 \hline
\end{tabular}
\end{center}
\end{table}

\begin{table}
\begin{center}
\small
\begin{tabular}{|c|c|c|c|c|}
 \hline
$p=100, n=50$ &  Banded & Block & Toeplitz & Random \\ 
\hline 
Lasso $L_\infty$  & 2.64 (0.205) & 2.31 (0.093) & 2.49 (0.308) & 2.40 (0.114)\\  
 
Lasso  $L_F$ & 2.73 (0.168) & 2.44 (0.141) & 2.74 (0.125) & 2.48 (0.119)\\  

Robust-Lasso  $L_\infty$ & 2.65 (0.073) & 2.26 (0.101) &  2.67 (0.039) & 2.18 (0.084)\\  

Robust-Lasso $L_F$  & 2.67 (0.080) & 2.38 (0.139) & 2.69 (0.052)& 2.32 (0.131)\\ 
 \hline
 \hline
Dantzig $L_1$  & 3.13 (0.177) & 2.70 (0.146) &  3.15 (0.140) & 3.21 (0.136)\\  
 
Dantzig  $L_F$ & 3.16 (0.073) & 3.06 (0.172)& 3.58 (0.116) & 3.75 (0.173)\\  

Robust-Dantzig  $L_1$ & 1.80 (0.069) & 1.82 (0.051)& 1.72 (0.047)& 1.51 (0.073)\\  

Robust-Dantzig $L_F$  & 2.78 (0.071) & 2.01 (0.104)& 2.77 (065)& 2.45 (0.090)\\ 
 \hline
\end{tabular}
\end{center}
\end{table}

\begin{table}
\begin{center}
\small
\begin{tabular}{|c|c|c|c|c|}
 \hline
$p=500, n=100$ &  Banded & Block & Toeplitz & Random \\ 
\hline 
Lasso $L_\infty$  & 4.99 (0.091) & 4.12 (0.043) & 4.27 (0.052) & 4.49  (0.019)\\  
 
Lasso  $L_F$ & 8.16 (0.070) & 7.98 (0.004) & 8.05 (0.021) & 7.82 (0.052)\\  

Robust-Lasso  $L_\infty$ & 4.80 (0.012) & 3.31 (0.015)& 3.55 (0.051)& 3.40 (0.017)\\  

Robust-Lasso $L_F$  &  7.51 (0.120) & 7.50(0.177)& 7.69 (0.158) & 6.69 (0.220)\\
 \hline
 \hline
Dantzig $L_1$  & 5.03 (0.070) & 5.64 (0.034) & 5.18 (0.055) & 5.43 (0.050)\\  
 
Dantzig  $L_F$ & 8.64 (0.169) & 9.03 (0.199) & 9.18 (0.222) & 8.43  (0.192)\\  

Robust-Dantzig  $L_1$ & 4.51 (0.030) & 4.50 (0.017)& 4.69 (0.037) & 4.69 (0.034)\\  

Robust-Dantzig $L_F$  & 7.11 (0.123) & 7.05 (0.102)& 7.09 (0.099)& 6.76 (0.122)\\ 
 \hline
\end{tabular}
\end{center}
\end{table}

From statistical perspective, the tables indicate that both of our robust estimation methods outperform the regular Lasso and Dantzig, when the innovation vectors have fat tail and the transition matrix enjoys a sparsity pattern. 
In a nutshell, our robust methods is more advantageous in tackling non-Gaussian time series.

\section{Concluding Remarks}
Time series regression arises in a wide range of disciplines. Conventional tools are inadequate when it involves high dimensional temporal dependent and fat-tailed data. In this paper, we develop a novel Bernstein inequality for high dimensional linear processes, with the help of which, we have made contributions towards the robust estimation theory of high dimensional time series regression in the presence of fat tails. The convergence rate depends
on the strength of temporal and cross-sectional dependence, the moment condition, the dimension and the sample size. We allow the dimension to increase exponentially with the sample size as a natural requirement of consistency. To perform statistical inference of the estimates 
such as hypothesis
testing and construction of confidence intervals, one needs to develop the deeper result in terms of asymptotic distributional theory. The latter is more challenging and we leave it as future work.

\bibliographystyle{apalike}	 
\begin{bibliography}{ref_bernstein}
\end{bibliography}

\newpage
\begin{appendices}
\section{Simulation on Time Series Regression}
In this section, we evaluate the finite sample performance of the modified $\ell_1$-regularized Huber estimator proposed in Section 3.1 compared with the regular Huber estimator. 
We generate the linear process $\{X_i\}_{i=1}^n$ from a VAR model,
$$X_{i+1}=AX_i+\epsilon_i,$$
where we consider a Toeplitz transition matrix $A = (\lambda^{|i-j|})$ with $\lambda = 0.5$ and further scaled by $2\lambda_{\max}(A)$ to ensure the stationarity of $(X_i)$.
The innovation vectors $\epsilon_{i}$ have i.i.d.~coordinates drawn from Student's $t$-distribution with $\text{df}=5$. We construct $\xi_i$ in (3.1) as
$$\xi_i=\sum_{k=0}^\infty b_k\eta_{i-k},$$
where $b_k=\rho^k$ with a $\rho$ drawn from $Unif(-0.8, 0.8)$ and $\eta_i$ follows a Student's $t$-distribution with $\text{df}=5$ and is independent of $\epsilon_i$. Recall the linear model
$$Y_i=X_i^\top\beta^*+\xi_i.$$
We choose $\beta^*=(1,1,\dots,1,0,0,\dots,0)^\top$ with $s$ elements of value 1 and $p-s$ elements of value 0 for $s=2\lfloor \log (p)\rfloor$. In weight function $$w(x)=\min\bigg\{1, \frac b{|Bx|_2}\bigg\},$$ 
we select $b=5, 15, 50, 100$ and $B=I_p$. The simulation results with different $n, p$ are summarized in Table \ref{ts-regression}. We observe that neither small $b$ nor large $b$ can be consistently beneficial. The weight function with a small $b$ shrinks the covariates too aggressively, hence discards too much information of the tail behavior of the linear process. Large $b$ makes the shrinkage less effective, hence approaches the Huber estimator.
\begin{table}
\small
\caption{Experiment results on time series regression.}
\vspace{-20pt}
\label{ts-regression}
\begin{center}
\begin{tabular}{ |c|cccc|}
\hline 
$(n, p)$ & $(100, 10)$  & $(100, 100)$ & $(100, 500)$  & $(100, 1000)$\\
 \hline
 Huber  & 0.77 (0.090) & 3.64 (0.120) & 4.37 (0.079)  & 5.65 (0.064)\\
 \hline
  Weighted Huber ($b=5$) & 0.80 (0.043) & 4.25 (0.069) & 4.55 (0.082) & 5.89 (0.041)\\
\hline
 Weighted Huber ($b=15$) & 0.68 (0.035) & 3.15 (0.092) & 4.15 (0.035) & 5.28 (0.161)\\
\hline
 Weighted Huber ($b=50$) & 0.87 (0.042) &  3.24 (0.155) & 4.00 (0.090) & 5.11 (0.049)\\
\hline
 Weighted Huber ($b=100$) & 0.70 (0.086) & 3.70 (0.086) & 4.26 (0.135) & 5.30 (0.077) \\
 \hline
\end{tabular}
\end{center}
\end{table}

\section{Proofs of Results in Section 2}
In this section, we provide the proofs of the results presented in Section 2. 
\begin{proof}[Proof of Theorem 2.1]
We first define the filtration $\{\mathcal{F}_i\}$ with the $\sigma$-field $\mathcal{F}_i=\sigma(\boldsymbol{\varepsilon}_i,\boldsymbol{\varepsilon}_{i-1},\dots)$, and the projection operator $P_j(\cdot)=\mE(\cdot|\mathcal{F}_j)-\mE(\cdot|\mathcal{F}_{j-1})$. Conventionally it follows that $P_j(G(X_i))=0$ for $j\geq i+1$. We can write
$$\sum_{i=1}^nG(X_i)-\mE G(X_i)=\sum_{j=-\infty}^n\Big(\sum_{i=1}^nP_j(G(X_i))\Big)=:\sum_{j=-\infty}^nL_j,$$
where $L_j=\sum_{i=1}^nP_j(G(X_i))$. By the Markov inequality, for any $\lambda>0$, 
\begin{eqnarray}
\label{P}
&&	\mP\bigg(\sum_{i=1}^nG(X_i)-\mE G(X_i)\geq 2x\bigg)
	\leq \mP\bigg(\sum_{j=-\infty}^{0}L_j\geq x\bigg)+\mP\bigg(\sum_{j=1}^nL_j\geq x\bigg)\cr
&\leq & \me^{-\lambda x}\mE\bigg[\exp\bigg\{\lambda\sum_{j=-\infty}^{0}L_j\bigg\}\bigg]+\me^{-\lambda x}\mE\bigg[\exp\bigg\{\lambda\sum_{j=1}^{n}L_j\bigg\}\bigg].
\end{eqnarray}
We shall bound the right-hand side of \eqref{P} with a suitable choice of $\lambda>0$. Observing that $\{L_j\}_{j\leq n}$ is a sequence of martingale differences with respect to $\{\mathcal{F}_j\}$, we firstly seek an upper bound on $\mE[\me^{\lambda L_j}\bigr|\mF_{j-1}]$. 
By the Lipschitz condition (2.6) and the boundedness of $G$, it follows that
\begin{align}\label{lipscond}
	|L_j|
		 &\leq \sum_{i=1\lor j}^n\min\left\{\left|\mE\left[G(X_i)\bigr|\mF_j\right]-\mE\left[G(X_i)|\mF_{j-1}\right]\right|, 2M\right\}\notag\\
		 &\leq\sum_{i=1\lor j}^n\min\left\{g^\top|A_{i-j}|\mE\left[|\boldsymbol{\varepsilon}_j-\boldsymbol{\varepsilon}_j'|\bigr|\mF_j\right], 2M\right\},
\end{align}
where $\varepsilon'_j$ is an i.i.d.~copy of $\boldsymbol{\varepsilon}_j$. 
For notational convenience, we denote
$b_i^\top=g^\top|A_i|$ and $\eta_j=\mE(|\boldsymbol{\varepsilon}_j-\boldsymbol{\varepsilon}_j'|\big|\mF_j)$. Then we have
\begin{align*}|L_j|&\leq{2M\sum_{i=1\lor j}^n\mb{I}(b_{i-j}^\top\eta_j\geq 2M)}+{\sum_{i=1\lor j}^n b_{i-j}^\top\eta_j\mb{I}(b_{i-j}^\top\eta_j\leq 2M)}=:I_j+II_j.
\end{align*}
For $j\leq 0$ and $k\geq2$, by the triangle inequality,  it holds that
\begin{eqnarray}\label{L_jbound}\mE[|L_j|^k\big|\mF_{j-1}]&\leq&\left[\left(\mE[|I_j|^k\big|\mF_{j-1}]\right)^{1/k}+\left(\mE[|II_j|^k\big|\mF_{j-1}]\right)^{1/k}\right]^k\cr
&\leq&\left(\|I_j\|_k+\|{II}_j\|_k\right)^k.\end{eqnarray}
Moreover,
\begin{align}\label{more}
\|{I}_j\|_k\leq2M\sum_{i=-j}^\infty \big\|\mb{I}(b_{i}^\top\eta_j\geq2M)\big\|_k
	&\leq 2M\sum_{i=-j}^\infty\Big[\mP\big((b_{i}^\top\eta_j)^2\geq (2M)^2\big)\Big]^{1/k}.	
\end{align}
Recall the definitions of $\gamma$ and $\tau$. We have
$|b_i|_1\leq \gamma\rho_0^{i/\tau},$
which implies
$$\mE[(b_i^\top\eta_j)^2]\leq2\sigma^2|b_i|_1^2\leq 2\gamma^2\sigma^2\rho_0^{2i/\tau}, \text{ for all }j.$$
By the Markov inequality, we obtain from \eqref{more} that for $k\geq2$,
\begin{align}\label{I}
	\|{I}_j\|_k\leq
	2M\left(\frac{\gamma\sigma}{\sqrt{2} M}\right)^{2/k}\frac{\rho_0^{-2j/k\tau}}{1-\rho_0^{2/k\tau}}.
\end{align}
In view of the fact $1-x\geq -x\log x$ for $x\in(0,1)$, we can further relax the bound in \eqref{I}. Applying the Stirling formula, for $k\geq2$, we can obtain
\begin{eqnarray*}
	\|I_j\|_k^k&\leq& k^k\tau^k\rho_0^{-2/\tau}\left(\frac M{\log(1/\rho_0)}\right)^k\left(\frac{\gamma\sigma}{\sqrt{2} M}\right)^{2}\rho_0^{-2j/\tau}\cr
	&\leq &\frac{1}{2\sqrt {2\pi}}\left(\frac{\gamma\sigma}{\rho_0 M}\right)^2k!\tau^k\left(\frac {\me M}{\log(1/\rho_0)}\right)^k\rho_0^{-2j/\tau}.
\end{eqnarray*}
Define the constants \begin{align*}
 	C_1=\frac{1}{2\sqrt {2\pi}}\rho_0^{-2},\quad\text{and}\quad
 	C_2=\frac {\me }{\log(1/\rho_0)}.
 \end{align*}
Then we can simply write
\begin{align}
\label{eq: I}\|I_j\|_k^k\leq C_1k!\tau^kC_2^kM^{k-2}\gamma^2\sigma^2\rho_0^{-2j/\tau}.
\end{align}
Analogously, for $k\geq2$, we can also get
\begin{align}\label{eq: II}
	\|{II}_j\|_k^k
	\leq \Big[\sum_{i=-j}^\infty \big\{\mE\big[(b_i^\top\eta_{j})^2\left(2M\right)^{k-2}\big]\big\}^{1/k}\Big]^k
	\leq C_1k!\tau^kC_2^kM^{k-2}\gamma^2\sigma^2\rho_0^{-2j/\tau}.
\end{align}
By \eqref{L_jbound}, (\ref{eq: I}) and (\ref{eq: II}), we have
\begin{equation}\label{bd}
	\mE[|L_j|^k\big|\mF_{j-1}]\leq C_1k!\tau^k (C_2')^kM^{k-2}\gamma^2\sigma^2\rho_0^{-2j/\tau},
\end{equation}
where $C_2'=2C_2=2\me /\log(1/\rho_0)$. Now we are ready to derive an upper bound for $\mE[\me^{\lambda L_{j}}\big|\mF_{j-1}]$. By the Taylor expansion, we have
	$$\mE[\me^{\lambda L_j}|\mF_{j-1}]=1+\mE[\lambda L_j|\mF_{j-1}]+\sum_{k=2}^\infty\frac1{k!}\mE[\lambda^k L_j^k|\mF_{j-1}].$$
Notice that $\mE[L_j\bigr|\mF_{j-1}]=0$. For $0<\lambda<(C_2'M\tau)^{-1}$, we have
\begin{eqnarray*}
	\mE[\me^{\lambda L_{j}}\big|\mF_{j-1}]&\leq& 1+C_1M^{-2}\gamma^2\sigma^2\rho_0^{-2j/\tau}\sum_{k=2}^\infty \Big(C_2'M\tau\lambda\Big)^k\cr
	&\leq&\exp\left\{\frac{C_1'\gamma^2\sigma^2\tau^2\rho_0^{-2j/\tau}\lambda^2}{1-C_2'M\tau\lambda}\right\},
\end{eqnarray*}
where the constant $$C_1'=C_1(C_2')^2=\frac{1}{2\sqrt {2\pi}}\left(\frac{2\me}{\rho_0\log(1/\rho_0)}\right)^2,$$
Thus, recursively conditioning on $\mF_0,\mF_{-1},\dots$, we have for $0<\lambda<(C_2'\tau)^{-1}$,
\begin{eqnarray*}
\mP\bigg(\sum_{j=-\infty}^{0}L_j\geq x\bigg)&\leq & \me^{-\lambda x}\mE\bigg[\exp\bigg\{\lambda\sum_{j=-\infty}^{0}L_j\bigg\}\bigg]\cr
&\leq &\me^{-\lambda x}\exp\left\{\frac{C_1'\gamma^2\sigma^2\tau^2(1-\rho_0^{2/\tau})^{-1}\lambda^2}{1-C_2'M\tau\lambda}\right\}.
\end{eqnarray*}
Specifically, choosing $\lambda=x[C_2'M\tau x+2C_1'\gamma^2\sigma^2\tau^2(1-\rho_0^{2/\tau})^{-1}]^{-1}$
yields
	\begin{align}\label{bernstein1}\mP\bigg(\sum_{j=-\infty}^{0}L_j\geq x\bigg)&\leq\exp\left\{-\frac{x^2}{4C_1'\gamma^2\sigma^2\tau^2(1-\rho_0^{2/\tau})^{-1}+2C_2'M\tau x}\right\}\notag\\
	&\leq\exp\left\{-\frac{x^2}{2C_1'\gamma^2\sigma^2\rho_0^{-2}(\log(1/\rho_0))^{-1}\tau^3+2C_2'M\tau x}\right\}\notag\\
	&=\exp\left\{-\frac{x^2}{C_1''\tau^3\gamma^2\sigma^2+2C_2'M\tau x}\right\},\end{align}
where $C_1''=2C_1'\rho_0^{-2}(\log(1/\rho_0))^{-1}$. We can deal with $L_j$ for $j\geq 1$ by similar arguments and obtain
\begin{align*}
	\mE[\me^{\lambda L_{j}}\big|\mF_{j-1}]\leq\exp\left\{\frac{C_1'\gamma^2\sigma^2\tau^2\lambda^2}{1-C_2'M\tau\lambda}\right\} \text{ for } j \geq 1.
\end{align*}
In a similar way as deriving \eqref{bernstein1}, it follows that
\begin{align}\label{bernstein2}
	\mP\left(\sum_{j=1}^{n}L_j\geq x\right)\leq \exp\left\{-\frac{x^2}{C_1''\gamma^2\sigma^2\tau^2n+2C_2'M\tau x}\right\}.
\end{align}
Combining (\ref{P}), \eqref{bernstein1} and \eqref{bernstein2}, we have
\begin{align*}\mP\Big(\sum_{i=1}^nG(X_i)-\mE[G(X_i)]\geq x\Big)\leq 2\exp\left\{-\frac{x^2}{4C_1''\tau^2(\tau\lor n)+4C_2'M\tau x}\right\},
\end{align*}
which implies (2.7) for $\tau \leq n$.
\end{proof}

\begin{proof}[Proof of Theorem 2.2] We follow the starting steps when proving Theorem 2.1. Without assuming $G$ bounded, we have
\begin{equation*}
 |L_j| \leq \sum_{i=1\lor j}^ng^\top|A_{i-j}|\mE\left[|\boldsymbol{\varepsilon}_j-\boldsymbol{\varepsilon}_j'|\bigr|\mF_j\right] = \sum_{i=1 \lor j}^n b_{i-j}^\top \eta_{j} =: d_j^\top \eta_j.
\end{equation*}
For $j \leq -\tau$, we have
\begin{equation}
\label{eq dj}
|d_j|_1 \leq 
\sum_{i=1}^n |b_{i-j}|_1\leq \gamma \frac{\rho_0^{1/\tau}}{1-\rho_0^{1/\tau}}\cdot \rho_0^{-j/\tau} \leq  (\log (1/\rho_0))^{-1} \gamma \tau \rho_0^{-j/\tau}.
\end{equation}
Note that
\begin{align}
\label{eq: lambda L}
    \mE[e^{\lambda|L_j|}|\mathcal{F}_{j-1}] \leq \mE[e^{\lambda d_j^\top\eta_j}|\mathcal{F}_{j-1}] = 
    \mE[e^{\lambda d_j^\top\eta_j}] \leq \mE[e^{\lambda d_j^\top (|\boldsymbol{\varepsilon}_j| + |\boldsymbol{\varepsilon}_j'|)}].
\end{align}
Let $\lambda^*=c_0 (\log (1/\rho_0))(\gamma\tau)^{-1}$ and $Y_j = \lambda^* d_j^\top (|\boldsymbol{\varepsilon}_j|+|\boldsymbol{\varepsilon}_{j}'|)\rho_0^{j/\tau}$. By (\ref{eq dj}) and (\ref{eq: lambda L}), it follows that for any $j \leq -\tau$, $\mE e^{Y_j} \leq \theta^2$ and
\begin{eqnarray*}
 \mE[e^{\lambda^*|L_j|}-1|\mathcal{F}_{j-1}] &\leq & \mE e^{Y_j \rho_0^{-j/\tau}} -1 = \int_{0}^\infty \rho_0^{-j/\tau} e^{x \rho_0^{-j/\tau}} \mP(Y_j \geq x) dx \cr
 &\leq & \int_{0}^\infty \rho_0^{-j/\tau} e^{x \rho_0^{-j/\tau}} e^{-x}\theta^2 dx \cr &\leq & \frac{\rho_0^{-j/\tau}\theta^2}{1-\rho_0^{-j/\tau}} \leq \frac{\rho_0^{-j/\tau}\theta^2}{1-\rho_0} .
\end{eqnarray*}
Since $\mE[L_j|\mathcal{F}_j]=0$, for any $0 < \lambda \leq \lambda^*$, 
\begin{eqnarray*}
\mE[e^{\lambda L_j}-1|\mathcal{F}_{j-1}] &=&  \mE[e^{\lambda L_j}-\lambda L_j -1|\mathcal{F}_{j-1}] \cr
&\leq & \mE[e^{\lambda |L_j|}-\lambda |L_j| -1|\mathcal{F}_{j-1}]\cr
&\leq & \mE[e^{\lambda^* |L_j|}-\lambda^* |L_j| -1|\mathcal{F}_{j-1}] \cdot \lambda^2/(\lambda^*)^2\cr
&\leq & \mE[e^{\lambda^* |L_j|} -1|\mathcal{F}_{j-1}] \cdot \lambda^2/(\lambda^*)^2,
\end{eqnarray*}
in view of $e^x-x \leq e^{|x|}-|x|$ for any $x$ and when $x>0$, $(e^{\lambda x}-\lambda x-1)/\lambda^2$ is increasing with $\lambda \in (0, \infty)$. Using $1+x\leq e^x$, we have
\begin{eqnarray*}
    \mE[e^{\lambda L_j}|\mathcal{F}_{j-1}] &\leq& 1+ \mE[e^{\lambda^*|L_j|}-1|\mathcal{F}_{j-1}] \cdot \lambda^2/(\lambda^*)^2 \cr
    &\leq & 1+ {C_1 \rho_0^{-j/\tau} \gamma^{2}\tau^2\theta^2 \lambda^2}\leq \exp\left\{ {C_1 \rho_0^{-j/\tau} \gamma^{2}\tau^2\theta^2 \lambda^2} \right\}.
\end{eqnarray*}
where $C =  c_0^{-2}(\log(1/\rho_0))^{-2}/(1-\rho)$, which implies that
\begin{align*}
	\mP\bigg(\sum_{j=-\infty}^{-\tau}L_j\geq x\bigg)&\leq \me^{-\lambda x}\mE\bigg[\exp\bigg\{\lambda\sum_{j=-\infty}^{-1}L_j\bigg\}\bigg]
	\leq \me^{-\lambda x}\exp\left\{{C_1\gamma^2\tau^3\theta^2\lambda^2}\right\}.
\end{align*}
with $C_1 = C (\log(1/\rho_0))^{-1}(\rho_0)^{-2}$.
For the cases when $j> -\tau$, we use the bound $|d_j|_1\leq (\rho_0\log (1/\rho_0))^{-1}\gamma \tau$ and obtain $ \mE[e^{\lambda L_j}|\mathcal{F}_{j-1}] \leq 1+ C_2 \gamma^2 \tau^2\theta^2 \lambda^2$ for $C_2 = C/\rho_0^2$ and
\begin{align}
	\mP\bigg(\sum_{j=-\tau+1}^{n}L_j \geq x \bigg)&\leq \exp\left\{-\lambda x + C_2 (n+\tau) \gamma^2 \tau^2 \theta^2\lambda^2\right\}.
\end{align}
Therefore (2.9) follows by choosing $$\lambda = \min\bigg\{\lambda^*, \  \frac{x}{2C_1\gamma^2\tau^3\theta^2}, \ \frac{x}{2C_2(n+\tau)\gamma^2\tau^2\theta^2},\bigg\}.$$
\end{proof}

By a slight modification of the Lipschitz condition \eqref{lipscond}, we can develop some ancillary results in Corollar \ref{corbern} and Corollary \ref{corbern2}, that can be useful in estimating time series regression models. The proof follows similarly from that of Theorem 2.1 without extra technical difficulty.



\begin{cor}\label{corbern}
Consider the same setting of the model as in Theorem 2.1. Let $G:\mb{R}^{2p}\to\mb{R}$ be a function with $|G(u)|\leq M$ for all $u\in\mb{R}^{2p}$. Suppose there exists a vector $g = (g_1, \ldots, g_{2p})^\top$ with $g_i\geq0$ for $1\leq i\leq 2p$ and $\sum_{i=1}^{2p} g_i =1$ such that
	\begin{equation*}
		|G(u)-G(v)|\leq \sum_{i=1}^{2p} g_i |u_i-v_i|,\text{ for all }u, v\in \mathbb{R}^{2p}.
		\end{equation*}
	Then for any $x>0$, we have
	\begin{align}\label{thm}\mP\Big(\sum_{i=1}^nG(X_{i},X_{i-1})-\mE G(X_{i},X_{i-1})\geq x\Big)&\leq 2\exp\bigg\{-\frac{x^2}{C_1' n\sigma^2\gamma^2\tau^2+C_2'\tau Mx}\bigg\}.
\end{align}



\begin{proof}[Proof of Corollary \ref{corbern}]
It follows from the fact that the $(2p)$-dimensional process $(X_i^\top, X_{i-1}^\top)^\top$ is also linear and satisfies the condition (2.3) with $\gamma$ multiplied by a constant depending on $\rho_0$ only. 
\end{proof}
\end{cor}

\begin{cor}\label{corbern2}
Consider the same setting of the model as in Theorem 2.1. Let $G:\mb{R}^{p}\to\mb{R}$ be a function with $|G(u)|\leq M$ for all $u\in\mb{R}^{p}$. Assume that
	\begin{equation*}
		|G(u)-G(v)|\leq |u-v|_2,\text{ for all }u, v\in \mathbb{R}^{p}.
		\end{equation*}
Assume that $\log p >1$ and $\tau \log p \leq n$. Then for any $x>0$, we have
\begin{eqnarray}
\label{thm2}
&&\mP\Big(\sum_{i=1}^nG(X_i)-\mE G(X_i)\geq x\Big)\cr
&\leq & 2\exp\bigg\{-\frac{x^2}{C_1'' n(\sigma^2\gamma^2+M^2)\tau^2(\log p)^2+C_2''\tau M (\log p) x}\bigg\}.
\end{eqnarray}
\end{cor}

\begin{proof}[Proof of Corollary \ref{corbern2}]
With a different Lipschitz condition on $G$, the step \eqref{lipscond} becomes
\begin{equation*}
|L_j| \leq \sum_{i=1\lor j}^n \min\{|A^{i-j} \eta_j|_2,2M\} \leq \sum_{i=1\lor j}^n \min\{\gamma \rho_0^{(i-j)/\tau}|\eta_j|_2,2M\}.
\end{equation*}
Note that $\mE|\eta_j|_2^2 \leq 2p \sigma^2$.
For $j \leq -n_0$ where $n_0 = \lceil \tau \log p/\log (1/\rho_0) \rceil $, by similar arguments in deriving (\ref{bernstein1}), it can be obtained that
	\begin{align}\mP\bigg(\sum_{j=-\infty}^{-n_0}L_j\geq x\bigg)
	\leq \exp\left\{-\frac{x^2}{C_1\tau^3+C_2M\tau x}\right\}.\end{align}
For $j > -n_0$, we have
\begin{equation*}
|L_j| \leq 2n_0M + \sum_{i = j+n_0}^\infty \min\{\gamma \rho_0^{(i-j)/\tau}|\eta_j|_2,2M\}.
\end{equation*}
Similarly as \eqref{bd}, we can get
\begin{eqnarray*}
\mE[|L_j|^k|\mathcal{F}_{j-1}] &\leq& 2^k[(2n_0M)^k + C_1'k!\tau^k (C_2')^kM^{k-2}\gamma^2\sigma^2]\cr
&\leq & C_3(C_4 n_0M)^kk!(1+M^{-2}\gamma^2\sigma^2),
\end{eqnarray*}
which further implies
\begin{equation*}
	\mE\Big[\exp\Big\{\lambda \sum_{j=-s+1}^nL_j\Big\}\Big]\leq\exp\bigg\{\frac{C_3C_4^2(M^2+\gamma^2\sigma^2) n_0^2(n_0+n)\lambda^2}{1- C_4n_0M\lambda}\bigg\},
\end{equation*}
and
\begin{align*}
	\mP\left(\sum_{j=-n_0+1}^{n}L_j\geq x\right)\leq \exp\left\{-\frac{x^2}{C_3'(M^2+\gamma^2\sigma^2)n_0^2 (n_0+n)+C_4' M\tau(\log p) x}\right\}.
\end{align*}
Then (\ref{thm2}) follows in view of $n_0 \leq C_{\rho_0} n$.
\end{proof}

\begin{proof}[Proof of Theorem 2.3]
Let $\hat{\mu}_j$ be the Huber estimator of $\mu_j$. Following similar arguments of proving Theorem 3.1 in \citet{zhang2019robust}, for
$$R_{nj}(a) = \sum_{i=1}^n [\phi_\nu(X_{ij}-a)-\mE \phi_{\nu}(X_{ij}-a)],$$
it can be obtained that for any $\delta > 0$ with $\nu^{-1}\delta \leq 1/2$,
\begin{equation*}
\mP(\hat{\mu}_j - \mu_j \geq \delta) \leq \mP(R_{nj}(\mu_j +\delta)\geq n(\delta - 4\nu^{-1}\mu_2^2)). 
\end{equation*}
By the Lipschitz continuity of the function $\phi_\nu$ and the uniform bound $|\phi_\nu(x)| \leq \nu$, applying Theorem 2.1 to $R_{nj}(\mu_j+\delta)$, it follows that 
\begin{equation*}
\mP(R_{nj}(\mu_j+\delta)\geq y) \leq 2\exp\bigg\{-\frac{y^2}{2C_1n\tau^2\gamma^2+C_2\tau \nu y}\bigg\}.
\end{equation*}
Then it follows that
\begin{equation*}
\mP(\hat{\mu}_j - \mu_j \geq \delta) \leq 2x
\end{equation*}
by letting $n(\delta -4\nu^{-1}\mu_2^2) = y = \tau\gamma\sqrt{2C_1n\log(1/x)} + C_2\tau \nu\log(1/x) $ for $0<x<1/\me$.
The requirement $\nu^{-1}\delta \leq 1/2$ is met if we choose $\nu = \frac{2\mu^*}{\sqrt{C_2 }}\sqrt{\frac{n}{\log (1/x)}}$ for any $\mu^* \geq \mu_2$ and impose the condition
\begin{equation*}
(\sqrt{2C_1C_2}\gamma/\mu_2+ 4C_2 )\tau\log(1/x)\leq n. 
\end{equation*}
For $\delta \leq \delta_n = (\sqrt{2C_1}\gamma + 4\sqrt{C_2}\mu^*)\tau\sqrt{\frac{\log(1/x)}{n}}$, we have $\mP(\hat{\mu}_j - \mu_j \geq \delta_n) \leq 2x.$ It can also be obtained that $\mP(\hat{\mu}_j - \mu_j \leq -\delta_n) \leq 2x$ similarly.
By letting $x = p^{-c-1}$, for some $c>0$, it follows that  
\begin{equation*}
\mP\Big(\max_{1\leq j\leq p}|\hat{\mu}_j -\mu_j|\geq \sqrt{c+1}(\sqrt{2 C_1}\gamma + 4\sqrt{C_2}\mu^*)\tau\sqrt{\frac{\log p}{n}}\Big) \leq 4p^{-c}.
\end{equation*}
which further implies (2.10).
\end{proof}

\section{Proofs of Results in Section 3}
This section includes all the proofs for the results on robust estimation of time series regressions presented in Section 3.

\subsection{Proofs of Results in Section 3}
Denote $L_n(\beta)=\frac{1}{n}\sum_{i=1}^n \Phi_{\nu}((Y_i-X_i^\top \beta)w(X_i))$ and $\phi_\nu(\cdot) = \Phi_\nu'(\cdot)$. Recall  $b_0 = b/\lambda_{\min}(B)$ and $\kappa(B)=\lambda_{\max}(B)/\lambda_{\min}(B)$.
\begin{lem}[Deviation bound]
\label{lem: 3.1.1}
Let Assumptions (A1) (A2) (A3) in Section 3.1 be satisfied. Let $\nu =c\sigma_{\eta}(n/\log p)^{1/2}$ and $\lambda = C b_0 \sigma_{\eta}(\log p/n)^{1/2}$ for a sufficiently large $C$, with probability at least $1-4p^{-c_1}$ for some $c_1>0$, it holds that
$|\nabla L_n(\beta^*)|_\infty \leq \lambda.$
\end{lem}
\begin{proof}
    Consider the first component $\nabla L_{n1}(\beta^*)$ of $\nabla L_n(\beta^*)$. We have 
    $$\nabla L_{n1}(\beta^*)=\frac1n\sum_{i=1}^n\phi_\nu(\xi_iw(X_i))X_{i1}w(X_i).$$
    Note that $|\phi_\nu(x) - \phi_\nu(y)| \leq |x-y|$ and $|\phi_\nu(\xi_iw(X_i))X_{i1}w(X_{i})|\leq \nu b_0$.
    Conditioned on $\{X_i\}_{i=1}^n$, 
    by Theorem 2.1, we have 
    $$\mP\big(|\nabla L_{n1}(\beta^*) - \mE[\nabla L_{n1}(\beta^*)] |\geq C' b_0 x \ |(X_i)_i\big)\leq4p^{-c},$$
    for $x = \sigma_{\eta}\sqrt{\log p/n} + \nu \log p/n$ and some constant $c>1$. Hence by a union bound, with probability at least $1-4 p^{-c_1}$ for $c_1 >0$, it holds that
\begin{equation*}
|\nabla L_{n}(\beta^*) - \mE[\nabla L_{n}(\beta^*)] |_\infty \leq C' b_0 x.
\end{equation*}
As $\mE|\phi_\nu(\xi_iw(X_i))|= \mE [|\xi_iw(X_i)|\mathbf{1}(|\xi_iw(X_i)|>\nu)] \leq C_{\rho}\sigma_{\eta}^2 \nu^{-1}$,
we have
\begin{align}\label{eq:s3.1.1}
       | \mE[\nabla L_{n1}(\beta^*)]| \leq   \mE|\nabla L_{n1}(\beta^*)| \leq C_{\rho} b_0\sigma_{\eta}^2 \nu^{-1}.
\end{align}
Therefore, choosing $\nu = c \sigma_{\eta}(n/\log p)^{1/2}$ and $\lambda = C b_0 \sigma_{\eta} \sqrt{\log p/n}$ ensures that $|\nabla L_n(\beta^*)|_\infty\leq \lambda$ with high probability.

\end{proof}

\begin{lem}[RSC condition]
\label{lem: 3.1.2}
Let Assumptions (A1) (A2) (A3) be satisfied. Assume $$b_0(b_0+\kappa(B)\gamma\sigma_{\varepsilon})\tau\sqrt s\sqrt{(\log p)^3/n}\to0.$$ We have the following holds uniformly for all $\beta$, such that $|\Delta|_2\leq\nu/(2b_0)$ and $|\Delta_{S^c}|_1\leq3|\Delta_S|_1$ with probability no less than $1-4p^{-c_2}$ that
\begin{align}\label{eq:s3.1.2}
    L_n(\beta)-L_n(\beta^*)-\nabla L_n(\beta^*)^\top (\beta-\beta^*) \geq \frac12 \lambda_{\min}(\mE[\frac{w^2(X_i)}2X_iX_i^\top]) |\beta-\beta^*|_2^2.
\end{align}
\end{lem}
\begin{proof}
    Denote $S = \text{supp}(\beta^*)$. We will show that with high probability, \eqref{eq:s3.1.2} holds uniformly over the set 
    $$\mathcal{C}:=\{\beta:|\beta-\beta^*|\leq\frac\nu{2b_0}, |\beta_{S^c}-\beta^*_{S^c}|_1\leq3|\beta_S-\beta^*_S|_1\},$$
    Let $\mathcal{T}(\beta,\beta^*)=L_n(\beta)-L_n(\beta^*)-\nabla L_n(\beta^*)^\top (\beta-\beta^*)$, then it follows the same argument as Appendix B.3 in \cite{loh2021scale} that 
    $$\mathcal{T}(\beta,\beta^*)\geq\frac1n\sum_{i=1}^n\frac12(w(X_i)X_i^\top(\beta-\beta^*))^2\mathbf{1}_{A_i},$$
    where $A_i=\{\xi_i\leq \nu/2\}$. Denote $\Gamma=\frac1n\sum_{i=1}^n\frac{w(X_i)^2}2X_iX_i^\top\mathbf{1}_{A_i}$. For any $u$ such that $|u|_2\leq 1$, we have 
    $$u^\top\Gamma u=\frac1{n}\sum_{i=1}^n\frac12(u^\top X_iw(X_i))^2\mathbf{1}_{A_i}.$$ 
    Notice that $\frac12|(u^\top xw(x))^2-(u^\top yw(y))^2|\leq b_0(\kappa(B)+1)|x-y|_2$ and $|(u^\top x w(x))^2|\leq b_0^2$.
    Conditioned on $\xi_i$, by Corollary \ref{corbern2} we have 
    $$\mP(|u^\top\Gamma u-\mE[u^\top\Gamma u]|\geq t |(\xi_i)_i)\leq 4\exp\{-c_3s\log p\},$$
    where $t = C b_0 (b_0+\kappa(B)\gamma\sigma_{\varepsilon})\tau\sqrt s\sqrt{(\log p)^3/n}$ for a sufficiently large $C$ such that $c_3 > 4$. Note that $t\to0$ by assumption. 
    Following the same spirit of the $\varepsilon$-net argument in lemma 15 of \cite{loh2012high}, we can obtain that 
    $$\big|v^\top\big(\Gamma-\mE\Gamma\big)v\big|\leq t,\ \forall\, v\in\mb{R}^{p},\ |v|_0\leq2s, \ |v|_2\leq1,$$
    holds with probability at least $$1-4\exp\big\{2s\log9+2s\log p-c_3{s\log p}\big\}\geq1-4p^{-c_2},$$
    provided that $p\to \infty$ and a sufficiently large $c_3$. By Lemma 12 in \cite{loh2012high}, it further implies that 
    \begin{equation}|v^\top(\Gamma-\mE\Gamma)v|\leq 27t\bigg(|v|_2^2+\frac{|v|_1^2}{s}\bigg), \ \forall v\in\mb{R}^{p}.
\end{equation}
Denote $\Delta=\beta-\beta^*$, then we have 
\begin{equation}\label{eq:s3.1.3}
    \mathcal{T}(\beta,\beta^*) \geq \Delta^\top\Gamma\Delta \geq \mE[\Delta^\top\Gamma\Delta] - 27t(|\Delta|_2^2+\frac{|\Delta|_1^2}s).
\end{equation}
Moreover, as $\mE|\xi_i|^2 \leq C_{\rho}\sigma_{\eta}^2$ and $\nu \to \infty$,
\begin{align*}
    \mE[\Delta^\top\Gamma\Delta] &= \mE[\frac{w^2(X_i)}2(\Delta^\top X_i)^2] \cdot \mP(|\xi_i|\leq\frac\nu2)\\
    &\geq \lambda_{\min}(\mE[\frac{w^2(X_i)}2X_iX_i^\top]) |\Delta|_2^2 \cdot \Big(1-\frac{4\mE|\xi_i|^2}{\nu^2}\Big)\\
    &\geq \frac34\lambda_{\min}(\mE[\frac{w^2(X_i)}2X_iX_i^\top])|\Delta|_2^2,
\end{align*}
Also, for $\beta\in \mathcal{C}$, $|\Delta|_2^2+\frac{|\Delta|_1^2}s\leq17|\Delta|_2^2$. By \eqref{eq:s3.1.3}, we conclude that
\begin{align*}
    \mathcal{T}(\beta,\beta^*) &\geq \Big(\frac34\lambda_{\min}(\mE[\frac{w^2(X_i)}2X_iX_i^\top]) - 459t\Big)|\Delta|_2^2\\
    &\geq\frac12\lambda_{\min}(\mE[\frac{w^2(X_i)}2X_iX_i^\top])|\Delta|_2^2
\end{align*}
\end{proof}

\begin{proof}[Proof of Theorem 3.1]

With Lemma \ref{lem: 3.1.1} and Lemma \ref{lem: 3.1.2}, the proof follows the same spirit as Appendix B.1 of \cite{loh2021scale} without extra technical difficulty.
\end{proof}

\subsection{Proofs of Results in Section 3.2}
We shall first prove Proposition 3.2.
\begin{proof}[Proof of Proposition 3.2]
If $\lambda_{\max}(A) <1$, for any $\epsilon >0$, the matrix $B=A/[\lambda_{\max}(A)+\epsilon]$ has spectral radius strictly less than 1. By Theorem 5.6.12 of \cite{golub2013matrix}, $B$ is convergent in the sense that $\lim\limits_{k \rightarrow \infty} B^k=0$. Thus, $\|B^k\|\rightarrow 0$ as $k\rightarrow \infty$ and there exists some $N=N(\varepsilon, A)$ such that $\|B^k\|<1$ for all $k \geq N$, which implies $\|A^k\| \leq [\lambda_{\max}(A)+\epsilon]^k$ for all $k \geq N$. Therefore, given the constant $0<\rho_0<1$ and with an arbitrarily small $\epsilon$ with $\lambda_{\max}(A)+\epsilon<1$, there must exist some finite $k$ such that $\|A^k\| \leq \rho_0$.
The proof of the converse is easier by the fact that $[\lambda_{\max}(A)]^k = \lambda_{\max}(A^k) \leq \|A^k\| \text{ for any }k$.
\end{proof}
To prove Theorem 3.3, we introduce some preparatory lemmas. Define $\widetilde L_j(\boldsymbol b)=n^{-1}\sum_{i=1}^n(\widetilde{X}_{ij}-{\boldsymbol b}^\top\widetilde{X}_{i-1})^2$ for $1\leq j\leq p$.
\begin{lem}\label{lem1}
	Let Assumption (B1) be satisfied. For $\nu\asymp \mu_q(n/\log p)^{1/2(q-1)}$ and 
	$\lambda\asymp \tau\gamma\mu_q(\|A\|_{\infty}+1)[(\log p)/n]^{1/2-1/2(q-1)}$, with probability at least $1-4p^{-c_1}$ for some $c_1 >0$, it holds that
	\begin{equation}\label{l4.1}\big|\widetilde L_j(\boldsymbol a_{j\cdot})\big|_\infty\leq\lambda,\,\text{for all }1\leq j\leq p.
	\end{equation}
\end{lem}

\begin{proof}[Proof of Lemma \ref{lem1}]
We consider the first component of $\nabla\widetilde L_j(\boldsymbol a_{j\cdot})$, denoted by $\nabla\widetilde L_{j1}(\boldsymbol a_{j\cdot})$. Other components can be manipulated analogously. Let $G(X_i,X_{i-1})=2(\widetilde X_{i1}-\widetilde X_{i-1}^\top \boldsymbol a_{j\cdot})\widetilde X_{(i-1)1}$, where $\tilde{X}_{(i-1)1}$ is the first element of $\tilde{X}_{(i-1)}$.
Then we can write
$$\nabla\widetilde L_{j1}(\boldsymbol a_{j\cdot})=\frac1n\sum_{i=1}^nG(X_i,X_{i-1}).$$
Notice that $|G|\leq2(\|A\|_\infty+1)\nu^2$ and $|G(u)-G(v)|\leq g^\top|u-v|$, where $|g|_1\leq 4(\|A\|_\infty+1)\nu$. 
By Corollary \ref{corbern}, for $x=c'\gamma\tau\sqrt{(\log p)/n}$, we have
\begin{align}\label{l4.2}\mP\bigg(\Big|\nabla\widetilde L_{j1}(\boldsymbol a_{j\cdot})-\mE\big[\nabla\widetilde L_{j1}(\boldsymbol a_{j\cdot})\big]\Big|\geq 4\nu(\|A\|_\infty +1)x\bigg)&\leq 4\exp\Big\{-\frac{(c')^2\log p}{2C_1}\Big\}.\end{align}
In view of $\mE[\nabla L_n(\boldsymbol a_{j\cdot})]=0$, the triangle inequality and $|\widetilde{X}_{ij}| \leq |X_{ij}|$, 
\begin{eqnarray} \label{eq: E}
\hspace{-3 mm}\big|\mE\big[\nabla\widetilde L_{j1}(\boldsymbol a_{j\cdot})\big]\big|&=&\big|\mE\big[\nabla\widetilde L_{j1}(\boldsymbol a_{j\cdot})\big]-\mE\big[\nabla L_{j1}(\boldsymbol a_{j\cdot})\big]\big|\cr
	&=& 2\mE\big[\big|(\widetilde X_{ij}-\boldsymbol a_{j\cdot}^\top\widetilde X_{i-1})\widetilde X_{(i-1)1}-(X_{ij}-\boldsymbol a_{j\cdot}^\top X_{i-1})X_{(i-1)1}\big|\big]\cr
	&\lesssim& \mE\big[\big|{X}_{(i-1)1}(\widetilde{X}_{ij}- X_{ij})\big|\big]+ \mE\big[\big|{X}_{ij}(X_{(i-1)1}-\widetilde X_{(i-1)1})\big|\big] \cr
	&&+ |\boldsymbol{a}_{j\cdot}|^\top \mE\big[\big|{X}_{(i-1)1}(\widetilde X_{i-1}-X_{i-1})\big|\big]\cr
	&& + |\boldsymbol{a}_{j\cdot}|^\top \mE\big[\big|{X}_{i-1}(\widetilde X_{(i-1)1}-X_{(i-1)1})\big|\big].
\end{eqnarray}
Since $|\widetilde{X}_{ij}-X_{ij}| \leq |X_{ij}|{\bf 1}\{|X_{ij}|\geq \nu\}$, by H\"{o}lder's inequality, we have
\begin{eqnarray*}
\mE\big[\big|{X}_{(i-1)1}(X_{ij}-\widetilde X_{ij})\big|\big] &\leq & \|\tilde{X}_{(i-1)1}\|_q\cdot\|\tilde{X}_{ij}-X_{ij}\|_{q/(q-1)}\cr
&\leq &\mu_q \|\tilde{X}_{ij}-X_{ij}\|_{q/(q-1)},
\end{eqnarray*}
where
\begin{equation*}
\|\tilde{X}_{ij}-X_{ij}\|_{q/(q-1)}^{q/(q-1)} \leq \mE|X_{ij}|^{q/(q-1)} {\bf 1}\{|X_{ij}|\geq \nu\} \leq \mu_q^{q} \nu^{-q(q-2)/(q-1)}.
\end{equation*}
It then follows that $\mE\big[\big|{X}_{(i-1)1}(X_{ij}-\widetilde X_{ij})\big|\big] \leq \mu_q^q \nu^{2-q}.$
Other terms in (\ref{eq: E}) can be dealt with similarly. With the choice of $\nu$, we can get
$
\big|\mE\big[\nabla\widetilde L_{j1}(\boldsymbol a_{j\cdot})\big]\big| 
\leq c\nu(\|A\|_\infty +1)x.$
Letting $\lambda=C\nu(\|A\|_\infty+1)x$ for a sufficiently large $C$ and $c'>2\sqrt{C_1}$, it follows from \eqref{l4.2} that
\begin{align*}\mP\bigg(\Big|\nabla\widetilde L_{j1}(\boldsymbol a_{j\cdot})\Big|\geq \lambda\bigg)&\leq 4\exp\Big\{-\frac{(c')^2\log p}{2C_1}\Big\}.\end{align*}
By the Bonferroni inequality, we have
\begin{equation*}\label{l1.1}\mP\bigg(\Big|\nabla\widetilde L_{j}(\boldsymbol a_{j\cdot})\Big|_\infty\geq \lambda,\,\text{for all }1\leq j\leq p\bigg)\leq 4p^{-c_1}
\end{equation*}
where $c_1=2^{-1}C_1^{-1}{(c')^2}-2>0.$
\end{proof}

Define a cone $C(S)=\{\Delta\in\mb{R}^{p}:|\Delta_{S^c}|_1\leq3|\Delta_S|_1\}$ for a subset $S\subseteq\{1,2,\dots,p\}$.
We shall verify a restricted eigenvalue (RE) condition on the set $C(S)$ in the lemma below.
\begin{lem}\label{lem2}
Let Assumptions (B1) and (B2) be satisfied. Choose $\nu\asymp \mu_q(n/\log p)^{1/(2q-2)}$. Then for all $\Delta\in C(S)$,
	\begin{align}\label{eq: lem2}\Delta^\top \nabla^2 \widetilde L_j(\boldsymbol a_{j\cdot})\Delta\geq \frac{1}2|\Delta|_2^2\end{align}	
holds with probability at least $1-4p^{-c_2}$ for some constant $c_2>0$.
\end{lem}

\begin{proof}[Proof of Lemma \ref{lem2}]
Denote $\widetilde X=( \widetilde X_0, \widetilde X_1,\dots, \widetilde X_{n-1})^\top$. Then $\nabla^2\widetilde L_j(\boldsymbol a_{j\cdot})=2{\widetilde X^\top\widetilde X}/{n}=:\Gamma$. We shall first show that with probability at least $1-4p^{-c_2}$ for some positive constant $c_2$, it holds that
\begin{equation}\label{l5.1}
	\big|v^\top\big(\Gamma-\mE\Gamma\big)v\big|\leq t,\ \forall\, v\in\mb{R}^{p},\ |v|_0\leq2s, \ |v|_2\leq1,
\end{equation}
where $t=c_1\mu_q\gamma\tau s^2({\log p}/{n})^{1/2-1/[2(q-1)]}.$ For any $u\in\mb{R}^p$ such that $|u|_2\leq1$ and $|u|_0\leq s$ hence $|u|_1\leq\sqrt s$, write 
$$u^\top(\Gamma-\mE\Gamma)u=2n^{-1}\sum_{i=0}^{n-1}(u^\top \widetilde X_i)^2-\mE(u^\top\widetilde X_i)^2=:n^{-1}\sum_{i=0}^{n-1}G(X_i)-\mE[G(X_i)].$$
Thus, for $G(X_i)=(u^\top \widetilde X_i)^2$, we have
$$|G(x)-G(y)|\leq 2|u^\top(x+y)\cdot u^\top(x-y)|\leq 4s\nu g^\top|x-y|,$$
where $|g|_1\leq1$.
Apply Theorem 2.1 to function $G(X_i)/(4s\nu)$ and we have for any fixed $u$ such that $|u|_2\leq1$ and $|u|_0\leq s$,
$$\mP\Big(\big|u^\top\big(\Gamma-\mE\Gamma\big)u\big|\geq t\Big)\leq4\exp\big\{-c_3{s^2\log p}\big\}.$$
Following the same spirit of the $\varepsilon$-net argument in lemma 15 of \cite{loh2012high}, we can obtain that \eqref{l5.1} holds with probability at least $$1-4\exp\big\{2s\log9+2s\log p-c_3{s^2\log p}\big\}\geq1-4p^{-c_2},$$
provided that $p\to \infty$ and a sufficiently large $c_3$ (or equivalently $c_1$). By Lemma 12 in \cite{loh2012high} and \eqref{l5.1}, it further implies that with probability greater than $1-4p^{-c_2},$
\begin{equation}\label{l5.2}|v^\top(\Gamma-\mE\Gamma)v|\leq 27t\bigg(|v|_2^2+\frac{|v|_1^2}{s}\bigg), \ \forall v\in\mb{R}^{p}.
\end{equation}
Note that when $\Delta\in C(S)$,
\begin{equation}
\label{eqq}
|\Delta|_1=|\Delta_S|_1+|\Delta_{S^c}|_1\leq4|\Delta_S|_1\leq4\sqrt s|\Delta_S|_2\leq4\sqrt s|\Delta|_2.
\end{equation}
Furthermore, some algebra delivers that
\begin{eqnarray}\label{l5.3}
\Delta^\top\mE\big[\Gamma\big]\Delta
		=2\mE[(\widetilde X_1^\top\Delta)^2]
		&\geq & 2\Big(\Delta^\top \mE[ X_1X_1^\top ]\Delta-\Delta^\top \mE[X_1X_1^\top-\widetilde X_1\widetilde X_1^\top]\Delta\Big)\cr
		&\geq& 2|\Delta|_2^2-2|\Delta|_1^2 \big|\mE[X_1X_1^\top-\widetilde X_1\widetilde X_1^\top]\big|_\infty.
\end{eqnarray}

For any $1\leq j,k\leq p$, by the triangle and H{\"o}lder's inequality, 
\begin{equation*}
|\mE \tilde{X}_{ij}\tilde{X}_{ik} - \mE X_{ij}X_{ik} | \leq |\mE (\tilde{X}_{ij}-X_{ij})\tilde{X}_{ik})|+ |\mE(\tilde{X}_{ik}-X_{ik}){X}_{ij})|.
\end{equation*}
We have
\begin{eqnarray*}
|\mE (\tilde{X}_{ij}-X_{ij})\tilde{X}_{ik})| &\leq & \|\tilde{X}_{ik}\|_q\cdot\|\tilde{X}_{ij}-X_{ij}\|_{q/(q-1)}\cr
&\leq &\mu_q \|\tilde{X}_{ij}-X_{ij}\|_{q/(q-1)},
\end{eqnarray*}
where
\begin{equation*}
\|\tilde{X}_{ij}-X_{ij}\|_{q/(q-1)}^{q/(q-1)} \leq \mE|X_{ij}|^{q/(q-1)} {\bf 1}\{|X_{ij}|\geq \nu\} \leq \mu_q^{q} \nu^{-q(q-2)/(q-1)}.
\end{equation*}
It then follows that $|\mE (\tilde{X}_{ij}-X_{ij})\tilde{X}_{ik}| \leq \mu_q^q \nu^{2-q}$. We can also deal with $|\mE(\tilde{X}_{ik}-X_{ik})X_{ij})|$ similarly. As a result, we have the bias
\begin{equation}
\label{maxnorm}
|\mE [\tilde{X}_{ij}\tilde{X}_{ik} - X_{ij}X_{ik}] | \leq 2 \mu_q^q \nu^{2-q} \leq C \mu_q^2 \big(\frac{\log p}{n}\big)^{\frac{1}{2}-\frac{1}{2q-2}}.
\end{equation}

By (\ref{eqq}), (\ref{l5.3}) and (\ref{maxnorm}), it follows that 
	\begin{align}\label{l5.4}
		\Delta^\top\mE\big[\Gamma\big]\Delta
		&\geq2|\Delta|_2^2-16 Cs\mu_q^2\big(\frac{\log p}n\big)^{\frac12 - \frac1{2q-2}}|\Delta|_2^2\geq |\Delta|_2^2.
	\end{align}
Recall that $t=c_1\mu_q\gamma\tau s^2(\log p/n)^{1/2-1/[2(q-1)]}\to0$ by Assumption (B2).
Combining \eqref{l5.2} and \eqref{l5.4}, we can establish the following RE condition
$$\nabla^2 L_j(\boldsymbol a_{j\cdot})\geq |\Delta|_2^2-27t(|\Delta|_2^2+|\Delta|_1^2/s)\geq|\Delta|_2^2-459t|\Delta|_2^2\geq\frac{1}{2}|\Delta|_2^2,$$
for all $\Delta\in C(S)$ with probability no less than $1-4p^{-c_2}$.
\end{proof}

\begin{proof}[Proof of Theorem 3.3]
	Let $\widehat\Delta_j=\widehat{\boldsymbol a}_{j\cdot}-{\boldsymbol a}_{j\cdot}$ for $j=1,\dots,p$. As the solution of (3.5), $\widehat{\boldsymbol a}_{j\cdot}$ satisfies
		$$ \widetilde L_j(\widehat{\boldsymbol a}_{j\cdot})+\lambda|\widehat{\boldsymbol a}_{j\cdot}|_1\leq \widetilde L_j(\boldsymbol a_{j\cdot})+\lambda|\boldsymbol a_{j\cdot}|_1,$$
	which together with convexity implies,
	\begin{equation}\label{t5.1}0\leq \widetilde L_j(\widehat{\boldsymbol a}_{j\cdot})-\widetilde L_j(\boldsymbol a_{j\cdot})-\langle \nabla \widetilde L_j(\boldsymbol a_{j\cdot}),\widehat\Delta_j\rangle\leq\lambda(|\boldsymbol a_{j\cdot}|_1-|\widehat{\boldsymbol a}_{j\cdot}|_1)+\big|\nabla \widetilde L_j(\boldsymbol a_{j\cdot})\big|_\infty|\widehat\Delta_j|_1.\end{equation}
	Denote by $A$ and $B$ the events in Lemma \ref{lem1} and Lemma \ref{lem2} respectively. Then $\mP(A\cap B)=1-\mP(A^c\cup B^c)\geq1-8p^{-c}$ for $c=\min\{c_1,c_2\}.$ Conditioned on the event $A$, \eqref{t5.1} implies
	\begin{eqnarray*}
	0&\leq&|\boldsymbol a_{j\cdot, S}|_1-|\widehat{\boldsymbol a}_{j\cdot, S}|_1-|\widehat{\boldsymbol a}_{j\cdot, S^c}|_1+\frac12|\widehat\Delta_j|_1\cr
	&\leq&|\widehat\Delta_{j, S}|_1-|\widehat\Delta_{j, S^c}|_1+\frac12|\widehat\Delta_j|_1=\frac32|\widehat\Delta_{j,S
}|_1-\frac12|\widehat\Delta_{j, S^c}|_1,
	\end{eqnarray*}
	which further implies $\widehat\Delta_j\in C(S)$ for all $1\leq j\leq p$. Conditioned on the event $B$, by \eqref{l4.1} and the second inequality in \eqref{t5.1}, we have
	\begin{align}
		\frac{1}{2}|\widehat\Delta_j|_2^2\leq \big(\lambda+\big|\nabla L_n(\boldsymbol a_{j\cdot})\big|_\infty\big)|\widehat\Delta_j|_1
		\leq6\sqrt s\lambda|\widehat\Delta_j|_2.
		\end{align}
	This immediately shows for all $1\leq j\leq p$
		\begin{equation}\label{eq: l2}|\widehat\Delta_j|_2\leq{12\sqrt s\lambda}\asymp\mu_q\gamma\tau(\|A\|_{\infty}+1)\sqrt{s}\bigg(\frac{\log p}n\bigg)^{\frac12-\frac{1}{2q-2}}\end{equation}
	as well as
	$$|\widehat\Delta_j|_1\lesssim \mu_q\gamma\tau s(\|A\|_{\infty}+1)\bigg(\frac{\log p}n\bigg)^{\frac12-\frac{1}{2q-2}}.$$	
	Hence, (3.6) follows in view of $\|\widehat A-A\|_{\infty} = \max_{j}|\widehat{\Delta}_j|_1$.
	Moreover, if we consider the estimation of $\textbf{Vec}(A)=({\boldsymbol a}_{1\cdot}^\top,{\boldsymbol a}_{2\cdot}^\top,\dots,{\boldsymbol a}_{p\cdot}^\top)^\top\in\mb{R}^{p^2}$ with the sparsity parameter $\mathcal{S}=\sum_{i=j}^ps_j$, by Assumption (B2') and similar arguments of verifying the RE condition in Lemma \ref{lem2}, (\ref{eq: lem2}) becomes
	$$2\Delta^\top\bigg[I_p\otimes\bigg(\frac{\widetilde X^\top\widetilde X}{n}\bigg)\bigg]\Delta\geq\frac{1}2|\Delta|_2^2,\quad\text{for all }\Delta\in\mb{R}^{p^2}.$$
	Thus, similarly as \eqref{eq: l2}, (3.7) follows.
	\end{proof}
Next we shall concern the robust Dantzig-type estimator.
 \begin{lem}\label{lemdan}
 	Let Assumption (B1) be satisfied. Choose the truncation parameter $\nu\asymp \mu_q(n/\log p)^{1/(2q-2)}$. Let $\lambda \asymp \mu_q \gamma\tau(\|A\|_1+1)[(\log p)/n]^{(q-2)/(2q-2)}$. Then with probability at least $1-8p^{-c'}$ for some constant $c'>0$, it holds that
	$$\|\widehat\Sigma_0-\Sigma_0\|_{\max}\leq \lambda_0\quad\text{and}\quad \|\widehat\Sigma_1-\Sigma_1\|_{\max}\leq \lambda_0.$$
\end{lem}

\begin{proof}[Proof of Lemma \ref{lemdan}]
	Let $\lambda_0=C\mu_q\tau\gamma [(\log p)/n]^{(q-2)/(2q-2)}$ for a sufficiently large constant $C$. Applying Theorem 2.1 to the $(m,l)$-th entry of $\widehat\Sigma_0$, we have
	\begin{align*}
		\mP\Big(\frac1n\Big|\sum_{i=1}^{n}\widetilde X_{im}\widetilde X_{il}-\mE\widetilde X_{im}\widetilde X_{il}\Big|\geq \lambda_0\Big)\leq 4\exp\Big\{-\frac{c^2\log p}{2C_1}\Big\}=4p^{-c^2/(2C_1)}.
	\end{align*}
	By \eqref{maxnorm} in the proof of Lemma \ref{lem2}, we see that
	\begin{align*}
		\Big|\mE \widetilde X_{im}\widetilde X_{il}-\mE X_{im}X_{il}\Big| \leq c \mu_q^2 \big(\frac{\log p}{n}\big)^{\frac{1}{2}-\frac{1}{2q-2}}\leq \lambda_0.\end{align*}
	Therefore,
	\begin{eqnarray*}
	  &&\mP\Big(\frac1n\Big|\sum_{i=1}^{n}\widetilde X_{im}\widetilde X_{il}-\mE[ X_{im}X_{il}]\Big|\geq \lambda_0\Big) \cr
	  &\leq & \mP\Big(\frac1n\Big|\sum_{i=1}^{n}\widetilde X_{im}\widetilde X_{il}-\mE[ \widetilde X_{im}\widetilde X_{il}]\Big|\geq C_2\lambda_0\Big)\leq 4p^{-C_3}
	\end{eqnarray*}
	for some $C_3>1$.
	Taking a union bound yields
	$$\mP(\|\widehat \Sigma_0-\Sigma_0\|_{\max}\geq\lambda_0)\leq 4p^{-c'},$$
	where $c'=C_3-1>0$.
By Corollary \ref{corbern}, similar arguments apply to $\widehat\Sigma_1$, which delivers $\|\widehat\Sigma_1-\Sigma_1\|_{\max}\leq\lambda_0$ with probability at least $1-4p^{-c'}$. In conclusion, it holds simultaneously that $\|\widehat\Sigma_0-\Sigma_0\|_{\max}\leq\lambda_0$ and $\|\widehat\Sigma_1-\Sigma_1\|_{\max}\leq\lambda_0$ with probability at least $1-8p^{-c'}$.
\end{proof}

\begin{proof}[Proof of Theorem 3.4]
	We first show that $A$ is feasible to the convex programming (3.8) for $\lambda=(\|A\|_1+1)\lambda_0$ with high probability. By the Yule-Walker equation and Lemma \ref{lemdan},  we have
	\begin{align*}
		\|\widehat\Sigma_0 A-\widehat\Sigma_1\|_{\max}&\leq\|\widehat\Sigma_0A-\Sigma_1\|_{\max}+\|\Sigma_1-\widehat\Sigma_1\|_{\max}\\
		&\leq\|\widehat\Sigma_0-\Sigma_0\|_{\max}\|A\|_1+\|\Sigma_1-\widehat\Sigma_1\|_{\max}\leq\lambda,
	\end{align*}
	with probability no less than $1-8p^{-c'}$.
Therefore, conditioned on the event in Lemma \ref{lemdan}, we conclude that
$|\widehat{\boldsymbol a}_{\cdot j}|_1\leq|{\boldsymbol a}_{\cdot j}|_1$ for all $j=1,\dots,p$ and hence
$\|\widehat A\|_1\leq \|A\|_1.$
Then we have
\begin{eqnarray*}
	\|\widehat A-A\|_{\max}&=&\|\Sigma_0^{-1}(\Sigma_0\widehat A-\widehat\Sigma_1+\widehat\Sigma_1-\Sigma_1)\|_{\max} \cr
	&\leq &  \|\Sigma_0^{-1}(\Sigma_0\widehat A-\widehat\Sigma_0\widehat A+\widehat\Sigma_0\widehat A-\widehat\Sigma_1)\|_{\max}+\|\Sigma_0^{-1}(\widehat\Sigma_1-\Sigma_1)\|_{\max}\cr
	&\leq & \|\Sigma_0^{-1}\|_1\|\Sigma_0-\widehat\Sigma_0\|_{\max}\|\widehat A\|_1+\|\Sigma_0^{-1}\|_1\|\widehat\Sigma_0\widehat A-\widehat\Sigma_1\|_{\max}\cr
	&&+\|\Sigma_0^{-1}\|_1\|\widehat\Sigma_1-\Sigma_1\|_{\max}.
\end{eqnarray*}
By Lemma \ref{lemdan} and the feasibility of $\widehat A$, we have
$$	\|\widehat A-A\|_{\max}\leq \|\Sigma_0^{-1}\|_1(\lambda_0\|A_1\|+\lambda+\lambda_0)\\
	=2\|\Sigma_0^{-1}\|_1\lambda.$$
Now we shall bound $\|\widehat A-A\|_1$ from above. Denote by $S_j$ the support of ${\boldsymbol a}_{\cdot j}$ for $j=1,\dots,p$. Then for any $1\leq j\leq p$, we have
\begin{align}\label{6.4}
	|\widehat{\boldsymbol a}_{\cdot j}-{\boldsymbol a}_{\cdot j}|_1
	&=\big|\widehat{\boldsymbol a}_{\cdot j, S_j}-{\boldsymbol a}_{\cdot j, S_j}\big|_1+\big|\widehat{\boldsymbol a}_{\cdot j}\big|_1-\big|\widehat{\boldsymbol a}_{\cdot j, S_j}\big|_1 \cr
	&\leq \big|\widehat{\boldsymbol a}_{\cdot j, S_j}-{\boldsymbol a}_{\cdot j, S_j}\big|_1+\big|{\boldsymbol a}_{\cdot j}\big|_1-\big|\widehat{\boldsymbol a}_{\cdot j, S_j}\big|_1\cr
	&\leq 2\big|\widehat{\boldsymbol a}_{\cdot j, S_j}-{\boldsymbol a}_{\cdot j, S_j}\big|_1
	\leq 4s^*\|\Sigma_0^{-1}\|_1\lambda.
\end{align}
Since \eqref{6.4} holds for all $1\leq j\leq p$, we conclude that
$$\|\widehat A-A\|_1\leq4s^*\|\Sigma_0^{-1}\|_1\lambda\lesssim \mu_q s^*\gamma\tau\|\Sigma_0^{-1}\|_1(\|A\|_1+1)\bigg(\frac{\log p}n\bigg)^{\frac12-\frac{1}{2q-2}}.$$
\end{proof}

\end{appendices}

\end{document}